\documentclass[final,leqno]{siamltex}
\usepackage{amsmath,amssymb,mathtools,bm}
\usepackage{graphicx,booktabs,array,float}
\usepackage{microtype}
\usepackage[bookmarksopen,bookmarksopenlevel=0,bookmarksdepth=2]{hyperref}
\usepackage{doi}

\numberwithin{equation}{section}
\allowdisplaybreaks[2]
\newtheorem{assumption}[theorem]{Assumption}
\newtheorem{remark}[theorem]{Remark}

\newcommand{\R}{\mathbb R}
\newcommand{\Ih}{I_h}
\newcommand{\dt}{\delta_t}
\newcommand{\norm}[2]{\lVert #1\rVert_{#2}}

\newcommand{\pt}[1]{P^{\rm tan}[#1]}
\newcommand{\pn}[1]{P^N[#1]}
\newcommand{\mass}[3]{\bigl(#2\cdot\widetilde N_h[#1],\,#3\cdot\widetilde N_h[#1]\bigr)^h_{\Gamma[#1]}}
\newcommand{\stiff}[3]{\bigl(\partial_{s[#1]}#2,\,\partial_{s[#1]}#3\bigr)_{\Gamma[#1]}}

\newcommand{\Md}{d_{\mathcal M}}
\newcommand{\dd}{\,\mathrm d}

\title{Convergence of the Original BGN Method for mean curvature flow}

\author{Yifei Li\thanks{Mathematisches Institut, Universit\"at T\"ubingen, Auf der Morgenstelle 10, 72076 T\"ubingen, Germany (yifei.li@mnf.uni-tuebingen.de).}}

\date{}

\begin{document}
\maketitle

\begin{abstract}
We prove that the original Barrett--Garcke--N\"urnberg (BGN) method for mean curvature flow of curves converges in the manifold distance with order $h^2$ for the time step $\tau=h^2$. The key new technique is backward error analysis, which stems from the numerical analysis of ordinary differential equations. By combining the expansions of the discrete velocity in space and time, we construct an approximation flow whose interpolation satisfies the BGN method with an improved defect, of order $h^4$ with tangential part of order $h^6$. This construction also identifies the tangential velocity of the BGN method and shows that $\tau=O(h^2)$ is the correct relation between $\tau$ and $h$. With the improved defect, we then prove $H^1$ superconvergence of order $h^4$ of the numerical solution to the approximation flow using techniques of evolving surface finite element methods, which yields the convergence in the manifold distance. Numerical experiments confirm the predicted orders of defects.
\end{abstract}



\pagestyle{myheadings}\markboth{Y.~Li}{Convergence of the original BGN method}

\section{Introduction}
The parametric finite element method of Barrett, Garcke, and N\"urnberg (BGN) was first introduced for the mean curvature flow and related geometric evolution equations of curves \cite{BGN2007a,BGN2007b}, and was later extended to surfaces \cite{BGN2008}. Its key feature is that it allows the tangential motion. The tangential velocity is therefore not given a priori, but is determined implicitly by the discrete equations. For curves, this implicit tangential motion leads to an asymptotically equidistributed mesh, so the method avoids mesh degeneration without any remeshing. It is unconditionally stable and requires only piecewise linear elements and one linear solve per time step. Owing to these properties, the BGN approach has been extended to a wide range of problems; see the review \cite{BGN2020} and the references therein. For example, it has been developed into efficient solvers for solid-state dewetting \cite{BaoJiangWangZhao2017,ZhaoJiangBao2020,ZhaoJiangBao2021}, and into structure-preserving methods that conserve the enclosed area or volume exactly while retaining energy stability, for surface diffusion, anisotropic surface diffusion, and axisymmetric geometric flows \cite{BaoZhao2021,LiBao2021,BaoJiangLi2023,BaoGarckeNurnbergZhao2022}. Numerical studies consistently exhibit convergence in manifold distance \cite{BaoJiangWangZhao2017,JiangSuZhang2024}. However, even for the original BGN method for mean curvature flow, a rigorous convergence analysis is still missing.

Evolving surface finite element methods (ESFEM) provide a powerful framework for proving convergence of numerical methods for geometric PDEs \cite{DziukElliott2013}. The theory was first developed for surfaces moving with a prescribed velocity \cite{DziukElliott2007,LubichMansourVenkataraman2013,Kovacs2018}. It was then extended to geometric flows whose velocity is given explicitly by geometric quantities, such as mean curvature flow \cite{KovacsLiLubich2019,Li2020,BaiLi2024} and Willmore flow \cite{KovacsLiLubich2021}, and further to methods whose tangential velocity is given implicitly, such as the minimal deformation rate approach \cite{HuLi2022,HuangLiTang2026}. These results cover fully discrete schemes \cite{KovacsLubich2018}. Convergence has also been obtained for BGN-type methods with additional stabilization or modified tangential motion \cite{BaiLi2025,BaiGarckeVeerapaneni2026,Bai2026}. The convergence proofs in this framework consist of two parts, stability and consistency. Stability is proved by energy estimates combined with geometric perturbation estimates between the discrete surface and the interpolated exact surface. Consistency means that the interpolated exact solution satisfies the numerical method up to a small defect, which is again bounded by geometric perturbation estimates.

However, none of these arguments applies to the original BGN method, for two reasons. First, the stability analysis requires the velocity, including its tangential part, to be uniquely determined, and it yields convergence of the parametrization in $L^2$ and $H^1$. For the BGN method, the tangential velocity is implicitly determined on both the mesh size $h$ and the time step $\tau$. Also, the convergence can only be expected in a parametrization-independent sense, such as the manifold distance. Second, the stability estimate relies on inverse inequalities to keep the numerical surface close to the interpolated exact surface in $W^{1,\infty}$. This requires the consistency defect to be small enough to compensate for the negative powers of $h$, which typically require finite elements of degree $k\geq2$, while the original BGN method uses linear elements.

The main contribution of this paper is the first convergence proof for the original BGN method for mean curvature flow, with the time step $\tau=h^2$, without any stabilization or modification of the tangential motion. Under a regularity assumption, we prove that the numerical polygons converge to the exact curves in the manifold distance,
\begin{equation*}
 \max_m\Md(\Gamma[X_h^m],\Gamma(t_m))\leq Ch^2.
\end{equation*}

Based on the framework of ESFEM, our main tool is backward error analysis, which is a classical tool in the numerical analysis of ordinary differential equations \cite{HairerLubichWanner2006}. Our idea is to use it to obtain an improved defect, such that the ESFEM analysis can be applied to the linear finite element. Instead of comparing the numerical solution with the interpolation of the exact solution $X$, we seek a modified flow such that its interpolation satisfies the BGN method with a smaller defect. In space, for the interpolation of a smooth curve to satisfy the BGN method with a small defect, its discrete velocity has to be of the form
\begin{equation*}
 V=V_0+hV_1+h^2V_2+\cdots.
\end{equation*}
In time, the discrete velocity is the difference quotient $(Y(t+\tau)-Y(t))/\tau$, which leads to a modified equation
\begin{equation*}
 Y_t=V^0+\tau V^1+\tau^2V^2+\cdots.
\end{equation*}
Combining the two expansions with a suitable relation between $\tau$ and $h$, we obtain a single expansion
\begin{equation*}
 Y_t=V^{[0]}+hV^{[1]}+h^2V^{[2]}+\cdots.
\end{equation*}
In this paper, we construct the first three terms $V^{[0]}$, $V^{[1]} = 0$, and $V^{[2]}$ explicitly. The defect of the interpolation of the exact solution $X(t)$ is of order $h^2$, with tangential defect of order $h^2$. In contrast, the approximation flow $Y^{[2]}$ with velocity $V^{[0]}+h^2V^{[2]}$ has an improved defect, of order $h^4$ with tangential part of order $h^6$. Moreover, the normal part of $V^{[0]}$ is $-\kappa\vec n$. And its tangential part is determined by an elliptic equation, which also shows that $\tau=O(h^2)$ is the correct relation between $\tau$ and $h$. Hence the approximation flow $Y^{[0]}$ with velocity $V^{[0]}$ differs from $X$ only by a tangential velocity, and therefore describes the same curves, $\Gamma[Y^{[0]}(t)]=\Gamma(t)$.

Section~\ref{sec:scheme} introduces the BGN method and the main result. Section~\ref{sec:tangent} proves the normal--tangential estimates and the geometric perturbation estimates on polygons. Section~\ref{sec:bea} expands the defect, constructs the approximation flows $Y^{[0]}$ and $Y^{[2]}$, and proves their improved defects. Section~\ref{sec:proof} proves the convergence, and Section~\ref{sec:numerics} verifies the defect orders and convergence rates numerically.

\section{The BGN method and main result}\label{sec:scheme}
\subsection{Finite element spaces on an interval}
Let $I=\R/\mathbb Z$. For an integer $J\geq3$, set $h=J^{-1}$ and $\rho_j=jh$, with periodic indices. The scalar space $S_h$ consists of continuous periodic functions that are affine on each interval $(\rho_{j-1},\rho_j)$. The interpolation operator $\Ih$ assigns prescribed nodal values to their piecewise affine interpolant; it also applies componentwise to vector fields and to products specified at the nodes.

For a scalar or vector nodal function, define
\begin{equation*}
 D_hv_j=\frac{v_j-v_{j-1}}h,\qquad \overline v_j=\frac{v_j+v_{j-1}}2,
\end{equation*}
\begin{equation*}
 (v,w)_{0,h}=h\sum_jv_j\cdot w_j,\qquad \norm v{0,h}^2=(v,v)_{0,h},\qquad |v|_{1,h}^2=h\sum_j|D_hv_j|^2,
\end{equation*}
\begin{equation*}
 \norm v{1,h}^2=\norm v{0,h}^2+|v|_{1,h}^2.
\end{equation*}
These definitions refer only to the uniform parameter grid. In particular, none depends on a curve or a numerical solution. We use the ordinary norms
\begin{equation*}
 \norm v0=\norm v{L^2(I)},\qquad \norm v1^2=\norm v{H^1(I)}^2=\norm v0^2+\norm{v_\rho}0^2.
\end{equation*}
For $v_h\in [S_h]^2$, $|v_h|_{1,h}=\|(v_h)_\rho\|_0$, and
\begin{equation*}
 \norm{v_h}0\sim\norm{v_h}{0,h},\qquad \norm{v_h}1\sim\norm{v_h}{1,h}.
\end{equation*}
The constants are independent of $h$, we use $\sim$ to denote the equivalence. We will freely use these equivalences, as well as the exact product formula
\begin{equation*}
 D_h(vw)_j=\overline v_jD_hw_j+(D_hv_j)\overline w_j.
\end{equation*}
The time levels are $t_m=m\tau$, with $\tau=h^2$, and
\begin{equation*}
 \dt v^{m+1}=\frac{v^{m+1}-v^m}{\tau}.
\end{equation*}

\subsection{Geometry-dependent finite element spaces}
Let $R$ denote clockwise rotation through $\pi/2$. Suppose there is a regular counterclockwise parametrization $X:I\times[0,T]\to\R^2$. The image is denoted by $\Gamma(t)=X(I,t)$. The unit tangent vector $\vec\tau$, outward unit normal $\vec n$, the arclength derivative $\partial_s$ are given by
\begin{equation*}
 \vec\tau=\frac{X_\rho}{|X_\rho|},\qquad \vec n=R\vec\tau,\qquad \partial_s=\frac1{|X_\rho|}\partial_\rho.
\end{equation*} The curvature is
\begin{equation}\label{eq:Frenet-relations}
 \partial_s\vec\tau=-\kappa\vec n,\qquad \partial_s\vec n=\kappa\vec\tau.
\end{equation}
Thus $\kappa>0$ on a unit circle.

The problem is a counterclockwise parametrization $X:I\times[0,T]\to\R^2$ satisfying
\begin{equation*}
 X_t=-\kappa[X]\vec n[X],\qquad X(\rho,0)=X_0(\rho).
\end{equation*}

The interpolation of $X$, $X_h = I_h X$, gives a regular embedded polygon, define
\begin{equation*}
 \Gamma_h[X_h]:=X_h(I),\qquad S_h(\Gamma_h[X_h])= \{v_h\circ(X_h)^{-1}:v_h\in S_h\}.
\end{equation*}

For the polygon $\Gamma_h[X_h]$, we further consider the unit tangent vector, the unit normal vector, and the arclength derivative as
\begin{equation*}
 \begin{gathered}
 \vec h_j[X_h]=hD_h X_j, \qquad \vec\tau_j[X_h]=\frac{\vec h_j[X_h]}{|\vec h_j[X_h]|},\\
 \vec n_j[X_h]=R\vec\tau_j[X_h], \qquad \partial_{s[X_h]} w|_{\vec{h}_j}=\frac{h D_h w_j}{|\vec h_j[X_h]|}.
 \end{gathered}
\end{equation*}

For functions belonging to $S_h(\Gamma_h[X_h])$, the mass lumping is
\begin{equation*}
 \begin{aligned}
 &(v,w)^h_{\Gamma_h[X_h]} =\frac12\sum_j|\vec h_j[X_h]|\bigl(v_j\cdot w_j+v_{j-1}\cdot w_{j-1}\bigr), \quad \norm v{0,\Gamma_h[X_h]}^2=(v,v)^h_{\Gamma_h[X_h]},\\
 & |v|_{1, \Gamma_h[X_h]}^2 =  (\partial_{s[X_h]}v,\partial_{s[X_h]}v)_{\Gamma_h[X_h]},\quad \norm v{1,\Gamma_h[X_h]}^2=\norm v{0,\Gamma_h[X_h]}^2 + |v|_{1, \Gamma_h[X_h]}^2.
 \end{aligned}
\end{equation*}
If
\begin{equation*}
 0<c\leq |\vec h_j[X_h]|/h\leq C,
\end{equation*}
then for $v_h\in S_h$, $\left\|v_h\right\|_{0, \Gamma_h[X_h]}\sim\left\|v_h\right\|_{0, h}$, $|v_h|_{1, \Gamma_h[X_h]}\sim |v_h|_{1, h}$, $\left\|v_h\right\|_{1, \Gamma_h[X_h]}\sim\left\|v_h\right\|_{1, h}$.

Whenever $\vec h_j[X_h]+\vec h_{j+1}[X_h]\ne0$, the nodal tangent, nodal normal, and their unit versions are
\begin{equation}\label{eq:def of N,T}
 \begin{aligned}
 &\widetilde{T}_h[X_h](\rho_j)=\frac{\vec h_j[X_h]+\vec h_{j+1}[X_h]}{|\vec h_j[X_h]|+|\vec h_{j+1}[X_h]|}, \quad \widetilde{N}_h[X_h](\rho_j)=R\widetilde{T}_h[X_h](\rho_j),\\
 & T_h[X_h](\rho_j) = \frac{\widetilde{T}_h[X_h](\rho_j)}{|\widetilde{T}_h[X_h](\rho_j)|},    \quad N_h[X_h](\rho_j) = R T_h[X_h](\rho_j).
 \end{aligned}
\end{equation}
At each node, the normal component and the tangential component are
\begin{equation}
 \alpha_h=\Ih[v_h\cdot N_h[X_h]],\qquad \beta_h=\Ih[v_h\cdot T_h[X_h]],\qquad \alpha_h,\beta_h\in S_h.
\end{equation}
Define the finite element operators $\pn {X_h} = N_h[X_h]N_h[X_h]^{\mathsf T},\ \pt {X_h}=\mathrm I_2-\pn {X_h}:[S_h]^2\longrightarrow [S_h]^2$, they satisfy:
\begin{equation*}
 \begin{gathered}
 (\pn {X_h})^2=\pn {X_h},\qquad (\pt {X_h})^2=\pt {X_h},\\
 \pn {X_h}\pt {X_h}=\pt {X_h}\pn {X_h}=0,
 \end{gathered}
\end{equation*}
\begin{equation*}
 \pn {X_h}v_h=\Ih[\alpha_hN_h[X_h]],\quad \pt {X_h}v_h=\Ih[\beta_hT_h[X_h]].
\end{equation*}
\begin{equation*}
 v_h=\pn {X_h}v_h+\pt {X_h}v_h, \quad \norm{v_h}{0,h}^2=\norm{\pn {X_h}v_h}{0,h}^2+ \norm{\pt {X_h}v_h}{0,h}^2.
\end{equation*}

Given $X_0: [0, 1] \to \mathbb{R}^2$, the original BGN method is given as follows: let $X_h^0 = I_h(X_0)$ and $\Gamma^0 = X_h^0(I)$. For any $m = 0, 1, \ldots$, find $X_h^{m+1} \in [S_h]^2, \kappa_h^{m+1} \in S_h$, such that for any $\chi_h\in S_h, \phi_h\in [S_h]^2$
\begin{equation*}
 \left(\frac{X_h^{m+1} - X_h^m}{\tau}\cdot \vec n_h[X_h^m],\chi_h\right)^h_{\Gamma[X_h^m]} +\bigl(\kappa_h^{m+1},\chi_h\bigr)^h_{\Gamma[X_h^m]}=0,
\end{equation*}
\begin{equation*}
 \bigl(\kappa_h^{m+1}\vec n_h[X_h^m],\phi_h\bigr)^h_{\Gamma[X_h^m]} -\stiff{X_h^m}{X_h^{m+1}}{\phi_h}=0.
\end{equation*}
For positive edge lengths and averaged normals spanning $\R^2$, the BGN method is uniquely solvable; see \cite{BGN2007b,BGN2008,BGN2020}.

Eliminating curvature, and denote $V_h^{m+1} = \frac{X_h^{m+1} - X_h^m}{\tau}$ gives the equivalent BGN formulation:
\begin{equation}\label{eq:BGN formulation}
 \begin{aligned}
 &\left(|\widetilde{N}_h[X_h^m]|^2 \pn {X_h^m} V_h^{m+1} ,\pn {X_h^m} \phi_h\right)^h_{\Gamma[X_h^m]}\\
 & \quad+\stiff{X_h^m}{X_h^{m}}{\phi_h} + \tau\stiff{X_h^m}{V_h^{m+1}}{\phi_h}=0.
 \end{aligned}
\end{equation}

\subsection{Main theorem}
\begin{assumption}[Regularity of the exact solution]\label{ass:regular}
For regular initial curve parameterized by $X_0:I\to\R^2$, the solution for $X(t)$ is in $C^3([0,T]; C^{6}(I))$, with
\begin{equation*}
 \inf_{I\times[0,T]}|X_\rho|>0.
\end{equation*}
\end{assumption}

For simple closed curves $\Gamma_1,\Gamma_2$ enclosing $\Omega_1,\Omega_2$, define the manifold distance
\begin{equation}
 \Md(\Gamma_1,\Gamma_2)=|\Omega_1\mathbin\triangle\Omega_2|.
\end{equation}
This symmetric-difference-area convention for manifold distance is used, for example, in \cite{JiangSuZhang2024}.

\begin{theorem}[Convergence for $\tau=h^2$]\label{thm:main}
Under Assumption~\ref{ass:regular} and Assumption~\ref{ass:existence}, there are $h_0>0$ and $C>0$, independent of $h$, such that for any $h\leq h_0$
\begin{equation}\label{eq:main result}
 \max_m\Md(\Gamma[X_h^m],\Gamma(t_m))\leq Ch^2, \quad m\tau \leq T.
\end{equation}
where $X_h^m$ is the numerical solution of the BGN formulation \eqref{eq:BGN formulation} at time $t_m = m\tau$, with $\tau=h^2$.
\end{theorem}
\section{Normal--tangential estimates}\label{sec:tangent}
Let $X\in C^3(I)$ be a regular embedded parametrization and
\begin{equation}
 X_h=\Ih X,\qquad \inf_I|X_\rho|>0.
\end{equation}
Here $I_j=(\rho_{j-1},\rho_j)$ and $\vec\tau_h[X_h]|_{I_j}=\vec\tau_j[X_h]$, whereas $T_h[X_h],N_h[X_h]\in [S_h]^2$. Iterated differences refer to nodal values:
\begin{equation*}
 (D_h)^k v_{h,j}=h^{-k}\sum_{\ell=0}^k(-1)^\ell\binom{k}{\ell} v_h(\rho_{j-\ell}).
\end{equation*}
Differences of edge quantities are written explicitly. Constants are independent of $h$.

\subsection{Geometry of the interpolated curve}

\begin{lemma}\label{lem:referencegeometry}
For sufficiently small $h$, $\Gamma_h[X_h]$ is regular and embedded, and
\begin{equation}
 c\leq |D_hX_{h,j}|\leq C,\qquad |(D_h)^2X_{h,j}|+|(D_h)^3X_{h,j}|\leq C,
\end{equation}
\begin{equation}
 |\widetilde T_h[X_h](\rho_j)|\geq c,\qquad |D_hT_h[X_h]_j|+|(D_h)^2T_h[X_h]_j|\leq C.
\end{equation}
The edge quantities satisfy
\begin{equation}
 \begin{aligned}
 |\vec\tau_{j+1}[X_h]-\vec\tau_j[X_h]|&\leq Ch, \quad |\vec\tau_{j+1}[X_h]-2\vec\tau_j[X_h]+\vec\tau_{j-1}[X_h]|\leq Ch^2,\\
 \big||\vec h_{j+1}[X_h]|-|\vec h_j[X_h]|\big|&\leq Ch^2, \quad \big||\vec h_{j+1}[X_h]|-2|\vec h_j[X_h]|+|\vec h_{j-1}[X_h]|\big| \leq Ch^3.
 \end{aligned}
\end{equation}
Moreover,
\begin{equation}
 |\overline{T_h[X_h]}_j|\geq c,\qquad |\overline{T_h[X_h]}_j-\vec\tau_j[X_h]|\leq Ch,\qquad h\sum_j|D_hT_h[X_h]_j|^2\geq c.
\end{equation}
The corresponding frame bounds hold for $N_h[X_h]$ and $\vec n_h[X_h]$.
\end{lemma}

\begin{proof}
For $k=1,2,3$,
\begin{equation*}
 (D_h)^kX_{h,j} =h^{-k}\int_{[0,h]^k}\partial_\rho^kX(\rho_j-s_1-\cdots-s_k)\,\mathrm ds_1\cdots\mathrm ds_k.
\end{equation*}
Normalization of the following averages gives the difference bounds:
\begin{equation*}
 \vec\tau_j[X_h]=\frac{D_hX_{h,j}}{|D_hX_{h,j}|},\qquad T_h[X_h](\rho_j)= \frac{D_hX_{h,j}+D_hX_{h,j+1}} {|D_hX_{h,j}+D_hX_{h,j+1}|}.
\end{equation*}
Furthermore,
\begin{equation*}
 \max_j|T_h[X_h](\rho_j)-\vec\tau(\rho_j)|\leq Ch^2,\qquad \|\partial_\rho T_h[X_h]-\partial_\rho\vec\tau\|_{L^\infty(I)} \leq Ch,
\end{equation*}
\begin{equation*}
 h\sum_j|D_hT_h[X_h]_j|^2 \longrightarrow\int_I|\partial_\rho\vec\tau|^2\,\mathrm d\rho>0,
\end{equation*}
since
\begin{equation*}
 \partial_\rho\vec\tau\equiv0 \quad\Longrightarrow\quad 0=\int_I X_\rho\,\mathrm d\rho =\vec\tau(0)\int_I|X_\rho|\,\mathrm d\rho\ne0.
\end{equation*}
Finally, $\|X_h-X\|_{W^{1,\infty}(I)}\leq Ch$ preserves the local graph charts and the separation of disjoint parameter arcs of $X$.
\end{proof}

\subsection{Normal and tangential estimates on one polygon}

\begin{lemma}[Tangential Poincar\'e inequality]\label{lem:tancontrol}
For $\alpha_h,\beta_h\in S_h$,
\begin{equation}
 \norm{\Ih[\alpha_hN_h[X_h]]}{1,h}\sim\norm{\alpha_h}{1,h},\qquad \norm{\Ih[\beta_hT_h[X_h]]}{1,h}\sim\norm{\beta_h}{1,h}.
\end{equation}
For $v_h,w_h\in [S_h]^2$,
\begin{equation}\label{eq:tangential poincare inequality}
 \norm{\pt{X_h}v_h}{1,h}\leq C|\pt{X_h}v_h|_{1,h},
\end{equation}
\begin{equation}\label{eq:tan vs normal}
 \begin{aligned}
 &\bigl|\stiff{X_h}{\pn{X_h}v_h}{\pt{X_h}w_h}\bigr|\\
 &\qquad\leq C\norm{\pn{X_h}v_h}{0,h}\norm{\pt{X_h}w_h}{1,h}.
 \end{aligned}
\end{equation}
\end{lemma}

\begin{proof}
For $\beta_h\in S_h$, the nodal product formula gives
\begin{equation*}
 \norm{\Ih[\beta_hT_h[X_h]]}{0,h}=\norm{\beta_h}{0,h}
\end{equation*}
\begin{equation*}
 \overline{T_h[X_h]}_j\cdot D_hT_h[X_h]_j=\frac{1}{2h}\left(|T_h[X_h](\rho_j)|^2 - |T_h[X_h](\rho_{j-1})|^2\right) = 0,
\end{equation*}
\begin{equation*}
 \begin{aligned}
 |D_h\Ih[\beta_hT_h[X_h]]_j|^2 ={}&|D_h\beta_{h,j}|^2|\overline{T_h[X_h]}_j|^2+|D_hT_h[X_h]_j|^2|\overline\beta_{h,j}|^2,
 \end{aligned}
\end{equation*}
Thus $\norm{\Ih[\beta_hT_h[X_h]]}{1,h}\sim\norm{\beta_h}{1,h}$. It also applies for $N_h[X_h]$.

Poincar\'e and Lemma~\ref{lem:referencegeometry} yield
\begin{equation}
 \begin{aligned}
 \norm{\beta_h}{0,h}^2 & \leq C \left(|\beta_h|_{1, h}^2 +\left|\int_I\beta_h\,\mathrm d\rho\right|^2 \right) \\
 &\leq C |\beta_h|_{1, h}^2 + C h\sum_j|D_hT_h[X_h]_j|^2 \left|\int_I\beta_h\,\mathrm d\rho\right|^2\\
 &  \leq C |\beta_h|_{1, h}^2 + \left( C h \sum_j |D_hT_h[X_h]_j|^2 |\overline{\beta}_{h, j}|^2 + C h \sum_j |D_hT_h[X_h]_j|^2 |\beta_h|_{1, h}^2\right) \\
 &\leq C|\beta_h|_{1,h}^2 + Ch\sum_j|D_hT_h[X_h]_j|^2|\overline\beta_{h,j}|^2 \leq C|\Ih[\beta_hT_h[X_h]]|_{1,h}^2.
 \end{aligned}
\end{equation}
Take $\beta_h=\Ih[v_h\cdot T_h[X_h]]$, we then obtain tangential Poincar\'e.

For \eqref{eq:tan vs normal}, take
\begin{equation*}
 \alpha_h=\Ih[v_h\cdot N_h[X_h]],\qquad \beta_h=\Ih[w_h\cdot T_h[X_h]].
\end{equation*}
Since $N_h[X_h]_j\cdot T_h[X_h]_j=0$, we obtain
\begin{equation*}
 \begin{aligned}
 &\stiff{X_h}{\pn{X_h}v_h}{\pt{X_h}w_h}\\
 & = \sum_j \alpha_{h, j} N_h[X_h]_j \cdot \left(-\beta_{h, j-1} \frac{T_h[X_h]_{j-1}}{|\vec{h}_j[X_h]|}-\beta_{h, j+1} \frac{T_h[X_h]_{j+1}}{|\vec{h}_{j+1}[X_h]|}\right) \\
 & = \sum_j \alpha_{h, j} N_h[X_h]_j \cdot \biggl(\left(\beta_{h, j} - h D_h\beta_{h, j} \right) \frac{hD_hT_h[X_h]_{j}}{|\vec{h}_j[X_h]|}\\
 & \qquad\qquad -\left(\beta_{h, j} + h D_h \beta_{h, j+1}\right) \frac{hD_hT_h[X_h]_{j+1}}{|\vec{h}_{j+1}[X_h]|}\biggr) \\
 &\leq C\norm{\alpha_h}{0,h}\norm{\beta_h}{1,h} \leq C\norm{\pn{X_h}v_h}{0,h}\norm{\pt{X_h}w_h}{1,h}.
 \end{aligned}
\end{equation*}
\end{proof}

\begin{lemma}[Norm equivalence]\label{lem:norm}
For $v_h\in [S_h]^2$,
\begin{equation}
 \norm{v_h}{1,h}^2\sim \norm{\pn{X_h}v_h}{1,h}^2+\norm{\pt{X_h}v_h}{1,h}^2,
\end{equation}
\begin{equation}
 \norm{v_h}{1,h}^2\sim \norm{\partial_{s[X_h]}v_h\cdot\vec n_h[X_h]}{L^2(\Gamma_h[X_h])}^2 +\norm{\pt{X_h}v_h}{1,h}^2+\norm{v_h}{0,h}^2.
\end{equation}
\end{lemma}

\begin{proof}
$P^N[X_h]+P^{\rm tan}[X_h]=\mathrm I_2$ give
\begin{equation*}
 \|v_h\|_{1,h} \le \|P^N[X_h]v_h\|_{1,h} +\|P^{\rm tan}[X_h]v_h\|_{1,h} \le C\|v_h\|_{1,h}.
\end{equation*}

For the second equivalence, Lemma~\ref{lem:referencegeometry} gives
\begin{equation*}
 |N_h[X_h]_j-\vec n_j[X_h]| + |N_h[X_h]_{j-1}-\vec n_j[X_h]| \le Ch.
\end{equation*}

Taking differences of the nodal projections across $I_j$, we obtain
\begin{equation*}
 \begin{aligned}
 &\left| D_h(P^N[X_h]v_h)_j - \bigl(D_hv_{h,j}\cdot\vec n_j[X_h]\bigr)\vec n_j[X_h] \right| \\
 &\quad\le \frac{2}{h}\Bigl( |N_h[X_h]_j-\vec n_j[X_h]|\,|v_{h,j}| + |N_h[X_h]_{j-1}-\vec n_j[X_h]|\,|v_{h,j-1}| \Bigr) \\
 &\quad\le C\bigl(|v_{h,j}|+|v_{h,j-1}|\bigr).
 \end{aligned}
\end{equation*}

Hence
\begin{equation*}
 \left\| \partial_\rho(P^N[X_h]v_h) - \bigl((v_h)_\rho\cdot\vec n_h[X_h]\bigr)\vec n_h[X_h] \right\|_0 \le C\|v_h\|_{0,h}.
\end{equation*}

Together with the first equivalence and the metric bounds, this yields
\begin{equation*}
 \begin{aligned}
 \|v_h\|_{1,h}^2 &\sim \|P^N[X_h]v_h\|_{1,h}^2 +\|P^{\rm tan}[X_h]v_h\|_{1,h}^2 +\|v_h\|_{0,h}^2 \\
 &\sim \|(v_h)_\rho\cdot\vec n_h[X_h]\|_0^2 +\|P^{\rm tan}[X_h]v_h\|_{1,h}^2 +\|v_h\|_{0,h}^2 \\
 &\sim \|\partial_{s[X_h]}v_h\cdot\vec n_h[X_h]\| _{L^2(\Gamma_h[X_h])}^2 +\|P^{\rm tan}[X_h]v_h\|_{1,h}^2 +\|v_h\|_{0,h}^2.
 \end{aligned}
\end{equation*}

\end{proof}

\subsection{Geometric perturbation}
Throughout this subsection, we assume that
\begin{equation*}
 Y_h\in [S_h]^2,\qquad \norm{Y_h-X_h}{1,h}\leq h^3.
\end{equation*}

\begin{lemma}[Perturbation of geometric quantities]\label{lem:tube} For sufficiently small $h$
\begin{equation}\label{eq:geometry perturbation}
 \begin{aligned}
 &\norm{T_h[Y_h]-T_h[X_h]}{0,h} + \norm{\pt{Y_h}-\pt{X_h}}{0,h} +\norm{\vec\tau_h[Y_h]-\vec\tau_h[X_h]}0\\
 &\qquad \leq C|Y_h-X_h|_{1,h},
 \end{aligned}
\end{equation}
with the same bounds for $N_h, P^N,  \vec{\tau}_h$. For $v_h\in [S_h]^2$,
\begin{equation}
 \begin{aligned}
 &\norm{(\pn{Y_h}-\pn{X_h})v_h}{1,h} +\norm{(\pt{Y_h}-\pt{X_h})v_h}{1,h}\\
 &\qquad\leq C\bigl(h^{5/2}\norm{v_h}{1,h} +h^{3/2}\norm{v_h}{0,h}\bigr).
 \end{aligned}
\end{equation}
\end{lemma}

\begin{proof}
The inverse inequality gives
\begin{equation*}
 \begin{aligned}
 |\vec h_j[Y_h]-\vec h_j[X_h]|=h|D_h(Y_h-X_h)_j|\le h^{1/2}|Y_h-X_h|_{1,h}\le h^{7/2},
 \end{aligned}
\end{equation*}
Thus for sufficiently small $h$, $|\vec h_j[Y_h]|\ge ch-h^{7/2}>0, \, |\vec h_j[Y_h]+\vec h_{j+1}[Y_h]|>0$, and $\vec{\tau}_h[Y_h], T_h[Y_h]$ are well-defined.

For $T_h$, we write it as a combination of $D_h(Y_h - X_h)$
\begin{equation*}
 \begin{aligned}
 &|T_h[Y_h](\rho_j)-T_h[X_h](\rho_j)| \leq \frac{|\bigl(\vec h_j[Y_h]+\vec h_{j+1}[Y_h]\bigr) -\bigl(\vec h_j[X_h]+\vec h_{j+1}[X_h]\bigr)|}{|\vec h_j[X_h]+\vec h_{j+1}[X_h]|} \\
 & \qquad + \frac{|\Bigl( |\vec h_j[X_h]+\vec h_{j+1}[X_h]| -|\vec h_j[Y_h]+\vec h_{j+1}[Y_h]| \Bigr)T_h[Y_h](\rho_j)|}{|\vec h_j[X_h]+\vec h_{j+1}[X_h]|} \\
 & \leq \frac{h(1+|T_h[Y_h](\rho_j)|)}{|\vec h_j[X_h]+\vec h_{j+1}[X_h]|} \left(|D_h(Y_h - X_h)_j| + |D_h(Y_h - X_h)_{j+1}|\right)\\
 & \leq C \left(|D_h(Y_h - X_h)_j| + |D_h(Y_h - X_h)_{j+1}|\right)
 \end{aligned}
\end{equation*}
Therefore
\begin{equation*}
 \begin{aligned}
 &|\pt{Y_h}(\rho_j)-\pt{X_h}(\rho_j)| \leq |(T_h[Y_h](\rho_j)-T_h[X_h](\rho_j)) (T_h[Y_h](\rho_j))^{\mathsf T}| \\
 & \qquad + |T_h[X_h](\rho_j)(T_h[Y_h](\rho_j)-T_h[X_h](\rho_j))^{\mathsf T}| \\
 & \leq C \left(|D_h(Y_h - X_h)_j| + |D_h(Y_h - X_h)_{j+1}|\right).
 \end{aligned}
\end{equation*}
Similarly,
\begin{equation*}
 \begin{aligned}
 |\vec\tau_j[Y_h]-\vec\tau_j[X_h]| & \leq \frac{|\vec{h}_j[Y_h] - \vec{h}_j[X_h]|}{|\vec h_j[X_h]|} + \frac{|(|\vec{h}_j[Y_h]| - |\vec{h}_j[X_h]|)\vec{\tau}_j[Y_h]|}{|\vec h_j[X_h]|} \\
 &\le \frac{ h|D_h(Y_h-X_h)_j| +\bigl||\vec h_j[X_h]|-|\vec h_j[Y_h]|\bigr| }{ |\vec h_j[X_h]| } \le C|D_h(Y_h-X_h)_j|.
 \end{aligned}
\end{equation*}
Squaring and summing, we obtain \eqref{eq:geometry perturbation}.

\begin{equation}
 \begin{aligned}
 &\|(\pt{Y_h}-\pt{X_h})v_h\|_{1,h}\\
 & = \|(\pt{Y_h}-\pt{X_h})v_h\|_{0, h} + \|D_h\left((\pt{Y_h}-\pt{X_h})v_h\right)\|_{0, h} \\
 & \leq \left\|(T_h[Y_h] - T_h[X_h])T_h[Y_h]^{\mathsf T} v_h\right\|_{0, h} + \left\|T_h[X_h](T_h[Y_h] - T_h[X_h])^{\mathsf T} v_h\right\|_{0, h}\\
 & \qquad + \left\|\overline{\pt{Y_h}-\pt{X_h}}D_hv_h\right\|_{0, h} + \left\|D_h(\pt{Y_h}-\pt{X_h})\overline{v_h}\right\|_{0, h} \\
 & \leq C \Bigl(\max_j \left\|T_h[Y_h](\rho_j) - T_h[X_h](\rho_j)\right\| + \max_j \left\|D_h(\pt{Y_h}-\pt{X_h})_j\right\|\Bigr)\\
 & \qquad \times\left\|v_h\right\|_{0, h} + C \max_j \left\|(\pt{Y_h}-\pt{X_h})(\rho_j) - (\pt{X_h})(\rho_j)\right\| \left\|v_h\right\|_{1, h} \\
 & \leq C\bigl(h^{5/2}\norm{v_h}{1,h}+h^{3/2}\norm{v_h}{0,h}\bigr)
 \end{aligned}
\end{equation}
The same estimates also hold for $P^N$.
\end{proof}

\begin{lemma}[Perturbation of the normal part]\label{lem:normalforce}
For $\phi_h\in [S_h]^2$,
\begin{equation}\label{eq:stiffness, id, cmp}
 \begin{aligned}
 &\stiff{Y_h}{Y_h}{\phi_h}-\stiff{X_h}{X_h}{\phi_h}\\
 &\quad-\bigl(\partial_{s[X_h]}(Y_h-X_h)\cdot\vec n_h[X_h], \partial_{s[X_h]}\phi_h\cdot\vec n_h[X_h]\bigr) _{\Gamma_h[X_h]}\\
 &\qquad\leq C\norm{(Y_h-X_h)_\rho}{L^\infty(I)} |Y_h-X_h|_{1,h}|\phi_h|_{1,h}.
 \end{aligned}
\end{equation}
\end{lemma}

\begin{proof}
Direct computation gives
\begin{equation*}
 \begin{aligned}
 &\stiff{Y_h}{Y_h}{\phi_h}-\stiff{X_h}{X_h}{\phi_h}\\
 &\quad-\bigl(\partial_{s[X_h]}(Y_h-X_h)\cdot\vec n_h[X_h], \partial_{s[X_h]}\phi_h\cdot\vec n_h[X_h]\bigr) _{\Gamma_h[X_h]}\\
 &= \int_I \left(\vec\tau_h[Y_h]-\vec\tau_h[X_h] -\frac{(Y_h-X_h)_\rho\cdot\vec n_h[X_h]}{|(X_h)_\rho|} \,\vec n_h[X_h]\right) \cdot(\phi_h)_\rho\,\mathrm d\rho \\
 &= \int_I \left(\frac{ \bigl|| (X_h)_\rho|-|(Y_h)_\rho|\bigr| }{ |(X_h)_\rho|\,|(Y_h)_\rho| } \bigl((Y_h-X_h)_\rho\cdot\vec n_h[X_h]\bigr) \,\vec n_h[X_h]\right) \cdot(\phi_h)_\rho\,\mathrm d\rho \\
 & \qquad +  \int_I \left(-\frac12|\vec\tau_h[Y_h]-\vec\tau_h[X_h]|^2 \,\vec\tau_h[X_h]\right) \cdot(\phi_h)_\rho\,\mathrm d\rho
 \end{aligned}
\end{equation*}
Since $|\vec\tau_j[Y_h]-\vec\tau_j[X_h]|\le C|D_h(Y_h-X_h)_j|$, we obtain \eqref{eq:stiffness, id, cmp}.
\end{proof}

\begin{lemma}[Perturbation of the tangential part]\label{lem:force}
For $\beta_h\in S_h$,
\begin{equation}\label{eq:stiffness, tang, cmp}
 \begin{aligned}
 &\stiff{Y_h}{Y_h}{\Ih[\beta_hT_h[Y_h]]} -\stiff{X_h}{X_h}{\Ih[\beta_hT_h[X_h]]}\\
 &\leq Ch^2|Y_h-X_h|_{1,h}\norm{\beta_h}{1,h}.
 \end{aligned}
\end{equation}
\end{lemma}

\begin{proof}
For any $Z_h \in [S_h]^2$, direct computation gives
\begin{equation}\label{eq:equidistribution}
 \begin{aligned}
 &\stiff{Z_h}{Z_h}{\Ih[\beta_hT_h[Z_h]]} =\sum_j\beta_{h,j} (\vec\tau_j[Z_h]-\vec\tau_{j+1}[Z_h])\cdot T_h[Z_h](\rho_j)\\
 &=\sum_j\beta_{h,j} a(\vec{h}_{j+1}[Z_h] + \vec{h}_{j}[Z_h], \vec{h}_{j+1}[Z_h] - \vec{h}_{j}[Z_h]).
 \end{aligned}
\end{equation}
Here the auxiliary term $a(\cdot, \cdot)$ is given as
\begin{equation}
 \begin{aligned}
 a(v, w) &= \left(\frac{v-w}{|v-w|} - \frac{v+w}{|v+w|}\right) \cdot \frac{v}{|v|} \\
 & = \frac{|v-w| - |v+w|}{4|v|}\left|\frac{v-w}{|v-w|} - \frac{v+w}{|v+w|}\right|^2.
 \end{aligned}
\end{equation}

From the definition, we know that $a(v, w) = O(|w|^3)$. Therefore, for $|v|\geq c>0$ and sufficiently small $|w|$, it holds uniformly that
\begin{equation}
 |\partial_v a|\leq C |w|^3, \quad |\partial_w a| + |\partial_w \partial_v a| \leq C |w|^2, \quad |\partial_w \partial_w a| \leq C |w|.
\end{equation}

For $0\leq \theta \leq 1$, set $v^\theta, w^\theta$ as
\begin{equation}
 \begin{aligned}
 & v^\theta_j = \frac{1}{h}\left(\vec{h}_{j+1}[X_h + \theta (Y_h - X_h)] + \vec{h}_{j}[X_h + \theta (Y_h - X_h)] \right), \\
 & w^\theta_j = \frac{1}{h}\left(\vec{h}_{j+1}[X_h + \theta (Y_h - X_h)] - \vec{h}_{j}[X_h + \theta (Y_h - X_h)] \right).
 \end{aligned}
\end{equation}
We already know that, $|v_j^\theta| \geq c>0,\, |w_j^\theta|\leq C h, \, |v_{j+1}^\theta - v_{j}^\theta| \leq C h, |w_{j+1}^\theta - w_j^\theta| \leq C h^2$. Consequently, we have
\begin{equation}
 \begin{gathered}
 |\partial_v a(v_j^\theta, w_j^\theta)|\leq C h^3, \qquad |\partial_w a(v_j^\theta, w_j^\theta)| \leq C h^2,\\
 |\partial_w a(v_{j+1}^\theta, w_{j+1}^\theta) - \partial_w a(v_{j}^\theta, w_{j}^\theta) | \leq C h^3.
 \end{gathered}
\end{equation}

Combining everything together, we obtain
\begin{equation}
 \begin{aligned}
 &\bigl( \partial_{s[Y_h]}Y_h, \partial_{s[Y_h]}I_h[\beta_hT_h[Y_h]] \bigr)_{\Gamma_h[Y_h]}- \bigl( \partial_{s[X_h]}X_h, \partial_{s[X_h]}I_h[\beta_hT_h[X_h]] \bigr)_{\Gamma_h[X_h]} \\
 & =  \int_0^1 d\theta\sum_j \beta_{h, j}  \frac{d}{d\theta}a(v_j^\theta, w_j^\theta) \, \\
 & = \int_0^1 d\theta\sum_j \beta_{h, j} \frac{1}{h}\partial_v a(v_j^\theta, w_j^\theta) \cdot \left(\vec{h}_{j+1}[Y_h]-\vec{h}_{j+1}[X_h] + \vec{h}_j[Y_h] - \vec{h}_{j}[X_h]\right) \\
 & \qquad + \int_0^1 d\theta\sum_j \beta_{h, j} \frac{1}{h}\partial_w a(v_j^\theta, w_j^\theta)\\
 & \qquad\qquad \cdot \left(\left(\vec{h}_{j+1}[Y_h]- \vec{h}_{j+1}[X_h]\right) - \left(\vec{h}_{j}[Y_h]  - \vec{h}_{j}[X_h]\right)\right) \\
 & \leq C h^2 \left\|\beta_h\right\|_{0, h} |Y_h - X_h|_{1, h}\\
 & \qquad - \int_0^1 d\theta\sum_j h D_h(\beta_h)_{j+1}\partial_w a(v_{j+1}^\theta, w_{j+1}^\theta)\cdot D_h(Y_h - X_h)_{j+1}\\
 & \qquad + \int_0^1 d\theta\sum_j \frac{\beta_{h, j}}{h}\left(- \partial_w a(v_{j+1}^\theta, w_{j+1}^\theta) + \partial_w a(v_j^\theta, w_j^\theta)\right)\cdot D_h(Y_h - X_h)_{j+1}\\
 & \leq C h^2 \left\|\beta_h\right\|_{1, h} |Y_h - X_h|_{1, h} .
 \end{aligned}
\end{equation}
\end{proof}

\section{Approximation flows and the improved defect}\label{sec:bea}

\subsection{Expansion of the BGN formulation}\label{subsec:weakexpansions}
Let $Y\in W^{6,\infty}(I)$ be a regular embedded counterclockwise parametrization with $\inf_I|Y_\rho|\geq c>0$, and set $Y_h=\Ih Y$. All geometric quantities without an indicated argument are evaluated on $Y$. We take $h$ sufficiently small that the averages of $\inf_I |Y_{h, \rho}| \geq c/2$. Throughout, $O(h^r)$ denotes a term bounded by $h^rC(c,\norm Y{W^{6,\infty}(I)})$ uniformly in $I$.

\begin{lemma}[Expansion of the mass-lumping on the interval]\label{lem:weakquadrature}
For $f\in W^{4,\infty}(I)$ and $\phi_h\in S_h$,
\begin{equation}\label{eq:interval-mass-expansion}
 \begin{aligned}
 \left|\int_I\Ih[f\phi_h]\dd\rho-\int_I f\phi_h\dd\rho +\frac{h^2}{12}\int_I f_{\rho\rho}\phi_h\dd\rho\right| \leq Ch^4\norm f{W^{4,\infty}(I)}\norm{\phi_h}{0,h}.
 \end{aligned}
\end{equation}
\end{lemma}

\begin{proof}
Let $\psi_j$ be the nodal basis function at $\rho_j$. We have
\begin{equation}\label{eq:hat-second-derivative}
 \int_I f_{\rho\rho}\psi_j\dd\rho =\frac{f(\rho_{j-1})-2f(\rho_j)+f(\rho_{j+1})}{h}.
\end{equation}
The quadrature rule
\begin{equation}\label{eq:hat-quadrature-rule}
 \int_{-1}^1(1-|\xi|)p(\xi)\dd\xi =\frac{p(-1)+10p(0)+p(1)}{12}
\end{equation}
is exact for polynomials $p$ of degree at most three. Scaling, and \eqref{eq:hat-second-derivative}--\eqref{eq:hat-quadrature-rule} give
\begin{equation}\label{eq:hat-quadrature-error}
 \left|hf(\rho_j)-\int_I f\psi_j\dd\rho +\frac{h^2}{12}\int_I f_{\rho\rho}\psi_j\dd\rho\right| \leq Ch^5\norm{\partial_\rho^4f}{L^\infty(I)}.
\end{equation}
Multiplying \eqref{eq:hat-quadrature-error} by $\phi_h(\rho_j)$ and summing, with $h\sum_j|\phi_h(\rho_j)|\leq\norm{\phi_h}{0,h}$, proves \eqref{eq:interval-mass-expansion}.
\end{proof}

\begin{lemma}\label{lem:Y-derivatives}
The following identities hold:
\begin{align}
 Y_\rho &= |Y_\rho|\,\vec{\tau}, \label{eq:Y1}\\
 Y_{\rho\rho} &= \bigl(|Y_\rho|\bigr)_\rho\,\vec{\tau} - |Y_\rho|^2\kappa\,\vec{n}, \label{eq:Y2}\\
 Y_{\rho\rho\rho} &= \Bigl(\bigl(|Y_\rho|\bigr)_{\rho\rho}-|Y_\rho|^3\kappa^2\Bigr)\,\vec{\tau} - \Bigl(3|Y_\rho|\bigl(|Y_\rho|\bigr)_\rho\,\kappa + |Y_\rho|^2\kappa_\rho\Bigr)\,\vec{n}. \label{eq:Y3}
\end{align}
\end{lemma}

\begin{proof}
Since $\partial_\rho = |Y_\rho|\,\partial_s$, \eqref{eq:Frenet-relations} gives $\vec{\tau}_\rho = -|Y_\rho|\kappa\,\vec{n}$ and $\vec{n}_\rho = |Y_\rho|\kappa\,\vec{\tau}$. Identity \eqref{eq:Y1} is the definition of $\vec{\tau}$. Differentiating \eqref{eq:Y1},
\begin{equation*}
 Y_{\rho\rho} = \bigl(|Y_\rho|\bigr)_\rho\,\vec{\tau} + |Y_\rho|\,\vec{\tau}_\rho = \bigl(|Y_\rho|\bigr)_\rho\,\vec{\tau} - |Y_\rho|^2\kappa\,\vec{n} ,
\end{equation*}
which is \eqref{eq:Y2}. Differentiating \eqref{eq:Y2},
\begin{equation*}
 \begin{aligned}
 Y_{\rho\rho\rho} &= \bigl(|Y_\rho|\bigr)_{\rho\rho}\,\vec{\tau} + \bigl(|Y_\rho|\bigr)_\rho\,\vec{\tau}_\rho - \bigl(|Y_\rho|^2\kappa\bigr)_\rho\,\vec{n} - |Y_\rho|^2\kappa\,\vec{n}_\rho\\
 &= \bigl(|Y_\rho|\bigr)_{\rho\rho}\,\vec{\tau} - |Y_\rho|\bigl(|Y_\rho|\bigr)_\rho\kappa\,\vec{n} - \Bigl(2|Y_\rho|\bigl(|Y_\rho|\bigr)_\rho\kappa + |Y_\rho|^2\kappa_\rho\Bigr)\,\vec{n} - |Y_\rho|^3\kappa^2\,\vec{\tau} ,
 \end{aligned}
\end{equation*}
which is \eqref{eq:Y3}.
\end{proof}

\begin{lemma}[Expansion of the nodal normal and tangential]\label{lem:frameexpansion}
For $\widetilde N_h[Y_h], T_h[Y_h]$ defined in \eqref{eq:def of N,T} and sufficiently small $h$,
\begin{equation}\label{eq:nodal-frame-expansion}
 \begin{aligned}
 &\left\|\widetilde N_h[Y_h]-\Ih\left[\vec n+h^2\Bigl(\Bigl(\frac12(|Y_\rho|)_\rho\kappa+\frac16|Y_\rho|\kappa_\rho\Bigr)\vec\tau -\frac18|Y_\rho|^2\kappa^2\vec n\Bigr)\right]\right\|_{L^\infty(I)}\\
 &\qquad \leq h^4C\!\left(c,\norm Y{W^{6,\infty}(I)}\right).
 \end{aligned}
\end{equation}
\begin{equation}\label{eq:nodal-tangent-expansion}
 \left\|T_h[Y_h]-\Ih\left[\vec\tau-h^2\Bigl(\frac12(|Y_\rho|)_\rho\kappa+\frac16|Y_\rho|\kappa_\rho\Bigr)\vec n\right]\right\|_{L^\infty(I)} \leq h^4C\!\left(c,\norm Y{W^{6,\infty}(I)}\right).
\end{equation}
\end{lemma}

\begin{proof}
It suffices to prove the estimate at the nodes.

For the numerator, Taylor's formula and \eqref{eq:Y1}, \eqref{eq:Y3}, together with $R\vec\tau=\vec n$ and $R\vec n=-\vec\tau$, give
\begin{equation}\label{eq:nodal-num}
 \begin{aligned}
 &\frac{R(\vec h_j[Y_h]+\vec h_{j+1}[Y_h])}{2h} =R\frac{Y(\rho_j+h)-Y(\rho_j-h)}{2h} =RY_\rho+\frac{h^2}{6}RY_{\rho\rho\rho}+O(h^4)\\
 &=|Y_\rho|\vec n+\frac{h^2}{6}\Bigl[\Bigl((|Y_\rho|)_{\rho\rho}-|Y_\rho|^3\kappa^2\Bigr)\vec n +\Bigl(3|Y_\rho|(|Y_\rho|)_\rho\kappa+|Y_\rho|^2\kappa_\rho\Bigr)\vec\tau\Bigr]+O(h^4).
 \end{aligned}
\end{equation}

For the denominator, Taylor's formula at $\rho_j\pm h/2$ gives
\begin{equation}\label{eq:Y at mid}
 \begin{aligned}
 \frac{|\vec h_{j+1}[Y_h]|}{h} &=\Bigl|Y_\rho+\frac{h^2}{24}Y_{\rho\rho\rho}\Bigr|(\rho_j+h/2)+O(h^4),\\
 \frac{|\vec h_{j}[Y_h]|}{h} &=\Bigl|Y_\rho+\frac{h^2}{24}Y_{\rho\rho\rho}\Bigr|(\rho_j-h/2)+O(h^4).
 \end{aligned}
\end{equation}
Since $|Y_\rho|\geq c > 0$, we have $\bigl|Y_\rho+\frac{h^2}{24}Y_{\rho\rho\rho}\bigr| =|Y_\rho|+\frac{h^2}{24}\vec\tau\cdot Y_{\rho\rho\rho}+O(h^4)$. By \eqref{eq:Y3}, $\vec\tau\cdot Y_{\rho\rho\rho}=(|Y_\rho|)_{\rho\rho}-|Y_\rho|^3\kappa^2$. Averaging the values at $\rho_j+h/2$ and $\rho_j-h/2$ cancels the odd powers of $h$, which yields
\begin{equation}\label{eq:nodal-den}
 \begin{aligned}
 \frac{|\vec h_j[Y_h]|+|\vec h_{j+1}[Y_h]|}{2h} &=\left(|Y_\rho| + \frac{h^2}{8}(|Y_\rho|)_{\rho\rho} + O(h^4)\right)\\
 &\quad+\left(\frac{h^2}{24}\Bigl((|Y_\rho|)_{\rho\rho}-|Y_\rho|^3\kappa^2\Bigr) + O(h^4)\right)\\
 &=|Y_\rho|+h^2\Bigl(\frac16(|Y_\rho|)_{\rho\rho}-\frac1{24}|Y_\rho|^3\kappa^2\Bigr)+O(h^4).
 \end{aligned}
\end{equation}

For any $x$ with $|Y_\rho|+x\geq c/2$, we have
\begin{equation*}
 \frac{1}{|Y_\rho|+x}=\frac1{|Y_\rho|}-\frac{x}{|Y_\rho|^2}+\frac{x^2}{|Y_\rho|^2(|Y_\rho|+x)}.
\end{equation*}
We apply this with $x=h^2\bigl(\frac16(|Y_\rho|)_{\rho\rho}-\frac1{24}|Y_\rho|^3\kappa^2\bigr)+O(h^4)$ and obtain
\begin{equation} \label{eq:cord expand}
 \frac{2h}{|\vec h_j[Y_h]|+|\vec h_{j+1}[Y_h]|} =\frac1{|Y_\rho|}-\frac{h^2}{|Y_\rho|^2}\Bigl(\frac16(|Y_\rho|)_{\rho\rho}-\frac1{24}|Y_\rho|^3\kappa^2\Bigr)+O(h^4).
\end{equation}
Multiplying this by \eqref{eq:nodal-num} gives
\begin{equation*}
 \begin{aligned}
 \widetilde N_h[Y_h](\rho_j) ={}&\Bigl(RY_\rho+\frac{h^2}{6}RY_{\rho\rho\rho}+O(h^4)\Bigr)\\
 &\times\Bigl(\frac1{|Y_\rho|}-\frac{h^2}{|Y_\rho|^2}\Bigl(\frac16(|Y_\rho|)_{\rho\rho}-\frac1{24}|Y_\rho|^3\kappa^2\Bigr)+O(h^4)\Bigr)\\
 ={}&\frac{RY_\rho}{|Y_\rho|} +h^2\Bigl(\frac{RY_{\rho\rho\rho}}{6|Y_\rho|} -\frac{RY_\rho}{|Y_\rho|^2}\Bigl(\frac16(|Y_\rho|)_{\rho\rho}-\frac1{24}|Y_\rho|^3\kappa^2\Bigr)\Bigr)+O(h^4)\\
 ={}&\vec n +h^2\Bigl(\frac{RY_{\rho\rho\rho}}{6|Y_\rho|} -\Bigl(\frac{(|Y_\rho|)_{\rho\rho}}{6|Y_\rho|}-\frac1{24}|Y_\rho|^2\kappa^2\Bigr)\vec n\Bigr)+O(h^4),
 \end{aligned}
\end{equation*}
where we used $RY_\rho=|Y_\rho|\vec n$ by \eqref{eq:Y1}. By \eqref{eq:Y3},
\begin{equation*}
 \frac{RY_{\rho\rho\rho}}{6|Y_\rho|} =\Bigl(\frac{(|Y_\rho|)_{\rho\rho}}{6|Y_\rho|}-\frac16|Y_\rho|^2\kappa^2\Bigr)\vec n +\Bigl(\frac12(|Y_\rho|)_\rho\kappa+\frac16|Y_\rho|\kappa_\rho\Bigr)\vec\tau .
\end{equation*}
Substituting this, the $(|Y_\rho|)_{\rho\rho}$-terms cancel, and we obtain
\begin{equation*}
 \widetilde N_h[Y_h](\rho_j) =\vec n+h^2\Bigl(\Bigl(\frac12(|Y_\rho|)_\rho\kappa+\frac16|Y_\rho|\kappa_\rho\Bigr)\vec\tau -\frac18|Y_\rho|^2\kappa^2\vec n\Bigr)+O(h^4),
\end{equation*}
which proves \eqref{eq:nodal-frame-expansion} at the nodes. And \eqref{eq:nodal-tangent-expansion} can be proved similarly using $T_h[Y_h] = -R\widetilde{N}_h[Y_h]/|\widetilde{N}_h[Y_h]|$.
\end{proof}

\begin{lemma}[Expansion of the mass-lumping]\label{lem:weakquadrature, mass-lumping}
For $u\in W^{4,\infty}(I)$ and $\phi_h\in S_h$,
\begin{equation}\label{eq:mass-expansion}
 \begin{aligned}
 &\Bigg|\mass{Y_h}{u_h}{\phi_h} -\int_I |Y_\rho|(u\cdot\vec n)\vec n\cdot\phi_h\,\dd\rho\\
 &\quad-\frac{h^2}{12}\int_I\Bigg[-\bigl(|Y_\rho|(u\cdot\vec n)\vec n\bigr)_{\rho\rho} + 2\bigl(|Y_\rho|\vec n\otimes\vec n\bigr)_{\rho\rho}u\\
 &\qquad\qquad +\frac{|Y_\rho|^3\kappa^2}{2} \bigl(9(u\cdot\vec n)\vec n-8u\bigr) \Bigg]\cdot\phi_h\,\dd\rho\Bigg|\\
 &\qquad\leq h^4C\!\left(c^{-1},\norm Y{W^{6,\infty}(I)}\right) \norm u{W^{4,\infty}(I)}\norm{\phi_h}{0,h}.
 \end{aligned}
\end{equation}
where $u_h = I_h u$.
\end{lemma}

\begin{proof}
First, use $\tilde{N}_h[Y_h](\rho_j) = \frac{Y(\rho_{j+1}) - Y(\rho_{j-1})}{|\vec{h}_{j-1}| + |\vec{h}_j|}$, we rewrite the mass-lumping term as an integration over $I$:
\begin{equation}
 \begin{aligned}
 &\mass{Y_h}{u_h}{\phi_h}\\
 & = h\sum_j I_h((u_h \cdot \tilde{N}_h[Y_h])( \phi_h \cdot \tilde{N}_h[Y_h])) (\rho_j) \frac{|\vec{h}_{j-1}| + |\vec{h}_{j}|}{2h} \\
 & = h\sum_j I_h\left(\left(u \cdot \frac{R(Y(\rho_{j+1}) - Y(\rho_{j-1}))}{2h}\tilde{N}_h[Y_h]\right) \cdot \phi_h \right)(\rho_j)
 \end{aligned}
\end{equation}
We have already known that
\begin{equation}
 \begin{aligned}
 &\frac{R(Y(\rho_{j+1}) - Y(\rho_{j-1}))}{2h}=  |Y_\rho|\vec{n}(\rho_j) + \frac{h^2}{6} RY_{\rho\rho\rho}(\rho_j) + O(h^4) \\
 &\widetilde{N}_{h}[Y_h] = \vec n+h^2\Bigl(\Bigl(\frac12(|Y_\rho|)_\rho\kappa+\frac16|Y_\rho|\kappa_\rho\Bigr)\vec\tau -\frac18|Y_\rho|^2\kappa^2\vec n\Bigr)+O(h^4).
 \end{aligned}
\end{equation}
Therefore, $u \cdot \frac{Y(\rho_{j+1}) - Y(\rho_{j-1})}{2h}\tilde{N}_h[Y_h]$ can be expanded as
\begin{equation}
 \begin{aligned}
 &u \cdot \frac{R(Y(\rho_{j+1}) - Y(\rho_{j-1}))}{2h}\tilde{N}_h[Y_h] - |Y_\rho|\left(u \cdot \vec{n}\right) \vec{n}\\
 & =  h^2\biggl( \left(\frac{1}{6}\left(u \cdot RY_{\rho\rho\rho}\right)- \frac{1}{8}u \cdot \vec{n} |Y_\rho|^3 \kappa^2\right) \vec{n}\\
 &\qquad + (u\cdot \vec{n})\Bigl(\frac{|Y_\rho|}{2}(|Y_\rho|)_\rho\kappa+\frac16|Y_\rho|^2\kappa_\rho\Bigr)\vec\tau\biggr) + O(h^4)
 \end{aligned}
\end{equation}

Finally, take $f = |Y_\rho| (u\cdot \vec{n}) \vec{n}$ in \eqref{eq:interval-mass-expansion}, we derive that
\begin{equation}
 \begin{aligned}
 &\int_{I} I_h(|Y_\rho|(u\cdot \vec{n})(\vec{n} \cdot \phi_h)) \dd \rho\\
 &\qquad = \int_{I} |Y_\rho|(u\cdot \vec{n})(\vec{n} \cdot \phi_h) - \frac{h^2}{12} \int_{I} (|Y_\rho| (u\cdot \vec{n}) \vec{n})_{\rho\rho} \phi_h + O(h^4).
 \end{aligned}
\end{equation}
Combine everything together and use \eqref{eq:Y3} proves \eqref{eq:mass-expansion}.
\end{proof}

\begin{lemma}[Expansion of the stiffness]\label{lem:weakflux}
For $u\in W^{6, \infty}(I)$, $\phi_h\in [S_h]^2$, we have
\begin{equation}\label{eq:stiffness-expansion}
 \begin{aligned}
 &\biggl|\stiff{Y_h}{\Ih u}{\phi_h} -\int_I\biggl(\frac{u_\rho}{|Y_\rho|}\\
 &\qquad +h^2\left(\frac{(|Y_\rho|)_\rho}{12|Y_\rho|^2}u_{\rho\rho} +\frac{|Y_{\rho\rho}|^2-3((|Y_\rho|)_\rho)^2} {24|Y_\rho|^3}u_\rho\right)\biggr)\cdot\phi_{h,\rho}\dd\rho\biggr|\\
 &\qquad\leq h^4C\!\left(c,\norm Y{ W^{6, \infty}(I)}\right) \norm u{ W^{6, \infty}(I)}\norm{\phi_h}{0,h}.
 \end{aligned}
\end{equation}
\end{lemma}

\begin{proof}
First, we rewrite the stiffness term as an integration over $I$:
\begin{equation}
 \begin{aligned}
 \stiff{Y_h}{\Ih u}{\phi_h} &= \sum_j \frac{1}{h}\int_{I_{j+1}}\frac{u(\rho_{j+1}) - u(\rho_{j})}{|Y(\rho_{j+1})-Y(\rho_{j})|}\\
 &\qquad \cdot \left(\phi_h(\rho_{j+1}) - \phi_h(\rho_{j})\right)\dd\rho.
 \end{aligned}
\end{equation}
Taylor expansion at $\rho_j+h/2$, gives
\begin{equation}
 \begin{aligned}
 \frac{1}{h}(u(\rho_{j+1}) - u(\rho_{j})) &= u_\rho(\rho_{j} + h/2)+h^2\frac{1}{24}u_{\rho\rho\rho}(\rho_{j} + h/2)+O(h^4),\\
 \frac{1}{h}(Y(\rho_{j+1}) - Y(\rho_{j})) &= Y_\rho(\rho_{j} + h/2)+h^2\frac{1}{24}Y_{\rho\rho\rho}(\rho_{j} + h/2)+O(h^4).
 \end{aligned}
\end{equation}
As in the proof of Lemma~\ref{lem:frameexpansion}, $|Y_\rho|\ge c$ and \eqref{eq:Y3} give
\begin{equation*}
 \Bigl|Y_\rho+\frac{h^2}{24}Y_{\rho\rho\rho}\Bigr| =|Y_\rho|+\frac{h^2}{24}\Bigl((|Y_\rho|)_{\rho\rho}-|Y_\rho|^3\kappa^2\Bigr)+O(h^4).
\end{equation*}
Thus
\begin{equation*}
 \begin{aligned}
 \frac{u(\rho_{j+1}) - u(\rho_{j})}{|Y(\rho_{j+1})-Y(\rho_{j})|} &=\frac{u_\rho}{|Y_\rho|}(\rho_{j} + h/2)\\
 &\quad +\frac{h^2}{24}\Bigl(\frac{u_{\rho\rho\rho}}{|Y_\rho|} -\frac{(|Y_\rho|)_{\rho\rho}}{|Y_\rho|^2}u_\rho+|Y_\rho|\kappa^2u_\rho\Bigr)(\rho_{j} + h/2)+O(h^4).
 \end{aligned}
\end{equation*}

Next, for a smooth function $f$ and a constant $g$ on $I_{j+1}$, Taylor expansion at $\rho_j+h/2$ gives
\begin{equation*}
 \frac1h\int_{I_{j+1}}f\,g\dd\rho=\frac{1}{h}\int_{I_{j+1}}\left(f(\rho_{j} + h/2)g+\frac{h^2}{24}f_{\rho\rho}(\rho)g\right)\dd\rho+O(h^4)\left\|f^{(5)}g\right\|_{L^{1}(I_{j+1})}.
\end{equation*}
Applying this to $f = u_\rho/|Y_\rho|$ and to $g = \phi_h(\rho_{j+1}) - \phi_h(\rho_{j}) = h \phi_{h, \rho}|_{I_{j+1}}$ yields
\begin{equation}
 \begin{aligned}
 &\frac{1}{h}\int_{I_{j+1}}\frac{u(\rho_{j+1}) - u(\rho_{j})}{|Y(\rho_{j+1})-Y(\rho_{j})|} \cdot \left(\phi_h(\rho_{j+1}) - \phi_h(\rho_{j})\right)\dd\rho \\
 & = \int_{I_{j+1}}\left(\frac{u_\rho}{|Y_\rho|} - \frac{h^2}{24} \left(\frac{u_{\rho}}{|Y_\rho|}\right)_{\rho\rho}\right) \cdot \phi_{h, \rho}\dd\rho + O(h^4) \\
 & \qquad + \int_{I_{j+1}}\frac{h^2}{24} \left(\frac{u_{\rho\rho\rho}}{|Y_\rho|} - \frac{(|Y_\rho|)_{\rho\rho}}{|Y_\rho|^2}u_\rho + |Y_\rho|\kappa^2u_\rho\right) \cdot \phi_{h, \rho}\dd\rho\\
 & = \int_{I_{j+1}}\biggl(\frac{u_\rho}{|Y_\rho|} + \frac{h^2}{24} \biggl(\frac{u_{\rho\rho\rho}}{|Y_\rho|} - \frac{(|Y_\rho|)_{\rho\rho}}{|Y_\rho|^2}u_\rho\\
 & \qquad\qquad + |Y_\rho|\kappa^2u_\rho - \left(\frac{u_{\rho}}{|Y_\rho|}\right)_{\rho\rho}\biggr)\biggr) \cdot \phi_{h, \rho}\dd\rho + O(h^4) \\
 & = \int_{I_{j+1}}\biggl(\frac{u_\rho}{|Y_\rho|} + \frac{h^2}{24} \biggl( |Y_\rho|\kappa^2u_\rho+\frac{2(|Y_\rho|)_\rho}{|Y_\rho|^2}u_{\rho\rho}\\
 & \qquad\qquad -\frac{2((|Y_\rho|)_\rho)^2}{|Y_\rho|^3}u_\rho\biggr)\biggr) \cdot \phi_{h, \rho}\dd\rho + O(h^4).
 \end{aligned}
\end{equation}
In the last equality we used $\Bigl(\frac{u_\rho}{|Y_\rho|}\Bigr)_{\rho\rho} =\frac{u_{\rho\rho\rho}}{|Y_\rho|}-\frac{2(|Y_\rho|)_\rho}{|Y_\rho|^2}u_{\rho\rho} +\frac{2((|Y_\rho|)_\rho)^2}{|Y_\rho|^3}u_\rho-\frac{(|Y_\rho|)_{\rho\rho}}{|Y_\rho|^2}u_\rho$. Summing over $j$, we obtain the desired result.
\end{proof}

\begin{lemma}[Tangential expansion of the stiffness, I]\label{lem:tangentexpansion}
For $\beta_h\in S_h$,
\begin{equation}\label{eq:tangential-geometric-expansion}
 \begin{aligned}
 &\biggl|\stiff{Y_h}{Y_h}{\Ih[\beta_hT_h[Y_h]]} +h^2\int_I \frac14|Y_\rho|(|Y_\rho|)_\rho\kappa^2\beta_h\dd\rho\\
 &\qquad - h^4 \int_I S_4[Y]\beta_h\dd\rho \biggr| \leq h^6C\!\left(c^{-1},\norm Y{W^{6,\infty}(I)}\right)\norm{\beta_h}{0,h},
 \end{aligned}
\end{equation}
where $S_4[Y]$ is the following fourth order tangential stiffness coefficient:
\begin{equation}\label{eq:tangential-geometric-coefficient-four}
 \begin{aligned}
 S_4[Y]={}& \frac{|Y_\rho|^3\kappa^3}{96} \bigl((|Y_\rho|)_\rho\kappa+2|Y_\rho|\kappa_\rho\bigr)+\frac{(|Y_\rho|)_{\rho\rho}\kappa}{48} \bigl((|Y_\rho|)_\rho\kappa+4|Y_\rho|\kappa_\rho\bigr)\\
 \qquad &+\frac{(|Y_\rho|)_\rho\kappa_\rho}{24} \bigl(|Y_\rho|\kappa_\rho-(|Y_\rho|)_\rho\kappa\bigr),
 \end{aligned}
\end{equation}
\end{lemma}

\begin{proof}
The identity \eqref{eq:equidistribution} in the proof of Lemma~\ref{lem:force} gives
\begin{equation}\label{eq:factored-tangential-load}
 \begin{aligned}
 &\stiff{Y_h}{Y_h}{\Ih[\beta_hT_h[Y_h]]} =\sum_j\beta_{h,j}(\vec\tau_j-\vec\tau_{j+1}) \cdot T_h[Y_h](\rho_j)\\
 &=-\frac{h^2}{4} \left(h\sum_j\beta_{h, j}\frac{2h}{|\vec h_j+\vec h_{j+1}|} \frac{|\vec h_{j+1}|-|\vec h_j|}{h^2} \frac{|\vec\tau_{j+1}-\vec\tau_j|^2}{h^2}\right).
 \end{aligned}
\end{equation}
At $\rho_j$, expand at $\rho_j \pm \frac{h}{2}$ and use \eqref{eq:Y at mid}  gives
\begin{equation}\label{eq:perturbation of length}
 \begin{aligned}
 \frac{|\vec h_j+\vec h_{j+1}|}{2h} &= |Y_\rho + \frac{h^2}{6}Y_{\rho\rho\rho} + O(h^4)| = |Y_\rho| + \frac{h^2}{6}Y_\rho \cdot Y_{\rho\rho\rho} + O(h^4) \\
 & = |Y_\rho| + \frac{h^2}{6}|Y_\rho|\left((|Y_\rho|)_{\rho\rho} -|Y_\rho|^3 \kappa^2\right)+ O(h^4).
 \end{aligned}
\end{equation}
\begin{equation}\label{eq:perturbation of h}
 \begin{aligned}
 \frac{|\vec h_{j+1}|-|\vec h_j|}{h^2}&=\frac{\left(|Y_\rho| + \frac{h^2}{24}((|Y_\rho|)_{\rho\rho} - |Y_\rho|^3 \kappa^2) + h^4 F(Y)\right)(\rho_j + \frac{1}{2}h) + O(h^6)}{h} \\
 & \qquad  \frac{-\left(|Y_\rho| + \frac{h^2}{24}((|Y_\rho|)_{\rho\rho} - |Y_\rho|^3 \kappa^2) + h^4 F(Y)\right)(\rho_j - \frac{1}{2}h) +  O(h^6) }{h} \\
 & = |Y_\rho|_\rho + \frac{h^2}{24} (|Y_\rho|_{\rho\rho\rho}) + \frac{h^2}{24}\left(((|Y_\rho|)_{\rho\rho} - |Y_\rho|^3 \kappa^2)\right)_\rho + O(h^4) \\
 & = |Y_\rho|_\rho + \frac{h^2}{24} \left(2(|Y_\rho|)_{\rho\rho\rho} -( |Y_\rho|^3\kappa^2)_\rho\right) + O(h^4).
 \end{aligned}
\end{equation}
\begin{equation*}
 \vec\tau_{j+1}=\Bigl(\vec\tau+\frac{h^2}{24}\Bigl(\frac{Y_{\rho\rho\rho}}{|Y_\rho|} -\frac{(Y_{\rho\rho\rho}\cdot Y_\rho)Y_\rho}{|Y_\rho|^3}\Bigr)\Bigr)(\rho_j+h/2)+O(h^4),
\end{equation*}
By \eqref{eq:Y3}, $\frac{Y_{\rho\rho\rho}}{|Y_\rho|}-\frac{(Y_{\rho\rho\rho}\cdot Y_\rho)Y_\rho}{|Y_\rho|^3} =-\bigl(3(|Y_\rho|)_\rho\kappa+|Y_\rho|\kappa_\rho\bigr)\vec n$. Taking the difference quotient at $\rho_j$ gives
\begin{equation*}
 \frac{\vec\tau_{j+1}-\vec\tau_j}{h} =\vec\tau_\rho+\frac{h^2}{24}\Bigl(\vec\tau_{\rho\rho\rho} -\bigl(\bigl(3(|Y_\rho|)_\rho\kappa+|Y_\rho|\kappa_\rho\bigr)\vec n\bigr)_\rho\Bigr)+O(h^4),
\end{equation*}
and therefore
\begin{equation*}
 \frac{|\vec\tau_{j+1}-\vec\tau_j|^2}{h^2} =|\vec\tau_\rho|^2+\frac{h^2}{12}\,\vec\tau_\rho\cdot\Bigl(\vec\tau_{\rho\rho\rho} -\bigl(\bigl(3(|Y_\rho|)_\rho\kappa+|Y_\rho|\kappa_\rho\bigr)\vec n\bigr)_\rho\Bigr)+O(h^4).
\end{equation*}
From $\vec\tau_\rho=-|Y_\rho|\kappa\vec n$ and $\vec n_\rho=|Y_\rho|\kappa\vec\tau$, we obtain $|\vec\tau_\rho|^2=|Y_\rho|^2\kappa^2$ and
\begin{equation*}
 \vec\tau_\rho\cdot\vec\tau_{\rho\rho\rho} =|Y_\rho|\kappa\,(|Y_\rho|\kappa)_{\rho\rho}-|Y_\rho|^4\kappa^4,
\end{equation*}
\begin{equation*}
 -\vec\tau_\rho\cdot\bigl(\bigl(3(|Y_\rho|)_\rho\kappa+|Y_\rho|\kappa_\rho\bigr)\vec n\bigr)_\rho =|Y_\rho|\kappa\,\bigl(3(|Y_\rho|)_\rho\kappa+|Y_\rho|\kappa_\rho\bigr)_\rho.
\end{equation*}
Moreover,
\begin{equation*}
 \begin{aligned}
 (|Y_\rho|\kappa)_{\rho\rho}&=(|Y_\rho|)_{\rho\rho}\kappa+2(|Y_\rho|)_\rho\kappa_\rho+|Y_\rho|\kappa_{\rho\rho},\\
 \bigl(3(|Y_\rho|)_\rho\kappa+|Y_\rho|\kappa_\rho\bigr)_\rho &=3(|Y_\rho|)_{\rho\rho}\kappa+4(|Y_\rho|)_\rho\kappa_\rho+|Y_\rho|\kappa_{\rho\rho}.
 \end{aligned}
\end{equation*}
Adding these two expressions, multiplying by $|Y_\rho|\kappa$ and subtracting $|Y_\rho|^4\kappa^4$ gives
\begin{equation}\label{eq:perturbation of tangent}
 \begin{aligned}
 &\frac{|\vec\tau_{j+1}-\vec\tau_j|^2}{h^2} - |Y_\rho|^2\kappa^2 \\
 &=\frac{h^2}{12} \left(4|Y_\rho|(|Y_\rho|)_{\rho\rho}\kappa^2+6|Y_\rho|(|Y_\rho|)_\rho\kappa\kappa_\rho +2|Y_\rho|^2\kappa\kappa_{\rho\rho}-|Y_\rho|^4\kappa^4\right)+O(h^4).
 \end{aligned}
\end{equation}

Substituting \eqref{eq:perturbation of length}-\eqref{eq:perturbation of tangent} in \eqref{eq:factored-tangential-load} proves \eqref{eq:tangential-geometric-expansion}.
\end{proof}

\begin{lemma}[Tangential expansion of the stiffness, II]\label{lem:tangentflux}
For $u\in W^{6, \infty}(I), \beta_h\in S_h$, we have
\begin{equation}\label{eq:tangential-stiffness-expansion}
 \begin{aligned}
 &\biggl|\stiff{Y_h}{u_h}{\Ih[\beta_hT_h[Y_h]]} +\int_I\left(\frac{u_\rho}{|Y_\rho|}\right)_\rho\cdot\vec\tau\,\beta_h\dd\rho\\
 &\qquad - h^2 \int_I S_2[u, Y]\beta_h\dd\rho \biggr| \leq h^4C\!\left(c^{-1},\norm Y{W^{6, \infty}(I)}\right) \norm u{W^{6, \infty}(I)}\norm{\beta_h}{0,h}.
 \end{aligned}
\end{equation}
where $S_2[u, Y]$ is the following second order tangential stiffness coefficient:
\begin{equation}\label{eq:tangential-stiffness-coefficient}
 \begin{aligned}
 S_2[u,Y]={}&- \left(\frac{(|Y_\rho|)_\rho}{12}\partial_{ss}u +\frac{|Y_\rho|^2\kappa^2}{24}\partial_su\right)_\rho \cdot\vec\tau\\
 &+\left(\frac12(|Y_\rho|)_\rho\kappa +\frac16|Y_\rho|\kappa_\rho\right) (\partial_su)_\rho\cdot\vec n\\
 &+\frac16(\partial_su)_{\rho\rho}\cdot\vec\tau_\rho +\frac1{12}(\partial_su)_\rho\cdot\vec\tau_{\rho\rho}.
 \end{aligned}
\end{equation}
\end{lemma}

\begin{proof}
Apply \eqref{eq:interval-mass-expansion} first to $f = \left(\frac{u_\rho}{|Y_\rho|}\right)_\rho$ with test $\Ih[\beta_h\vec\tau]$, and then to $f = \left(\frac{u_\rho}{|Y_\rho|}\right)_\rho \cdot \vec\tau$ with test $\beta_h$. Use \eqref{eq:interval-mass-expansion} in the terms multiplied by $h^2$. This gives
\begin{equation}\label{eq:tangential-product-quadrature}
 \begin{aligned}
 &\bigg|\int_I \left(\frac{u_\rho}{|Y_\rho|}\right)_\rho\cdot\Ih[\beta_h\vec\tau]\dd\rho -\int_I\left(\frac{u_\rho}{|Y_\rho|}\right)_\rho\cdot\vec\tau\beta_h\dd\rho\\
 &\hspace{14mm}+\frac{h^2}{12}\int_I \left(2\left(\frac{u_\rho}{|Y_\rho|}\right)_{\rho\rho}\cdot\vec\tau_\rho +\left(\frac{u_\rho}{|Y_\rho|}\right)_\rho\cdot\vec\tau_{\rho\rho}\right) \beta_h\dd\rho\bigg|\\
 &\qquad\leq h^4C\!\left(c^{-1},\norm Y{W^{6, \infty}(I)}\right) \norm u{W^{6, \infty}(I)}\norm{\beta_h}{0,h}.
 \end{aligned}
\end{equation}
Now take $\phi_h=\Ih[\beta_hT_h[Y_h]]$ in \eqref{eq:stiffness-expansion} and integrate by parts. Use \eqref{eq:nodal-tangent-expansion} from Lemma~\ref{lem:frameexpansion}, and $\partial_s u = \frac{u_\rho}{|Y_\rho|}, \partial_{ss} u = \frac{1}{|Y_\rho|} \left(\frac{u_\rho}{|Y_\rho|}\right)_\rho$, we obtain the quadratic coefficient $S_2[u, Y]$.
\end{proof}

\subsection{Defect and approximation flows}\label{subsec:weakassembly}

We consider a single, possibly $h$-dependent curve $Y$ determined by the following velocity field
\begin{equation}\label{eq:velocity-form}
 \begin{gathered}
 Y_t=V_0[Y]+h^2V_2[Y],\\
 V_0[Y]=\alpha_0[Y]\vec n+\beta_0[Y]\vec\tau, \qquad V_2[Y]=\alpha_2[Y]\vec n+\beta_2[Y]\vec\tau.
 \end{gathered}
\end{equation}
The four coefficients are determined by $Y$ and their regularity assumptions are given below.

Let $v(t)=\frac{Y(t+h^2)-Y(t)}{h^2}$ and $v_h=\Ih v(t)$, the defect is defined as
\begin{equation}\label{eq:assembled-defect-definition}
 \begin{aligned}
 d_h(\phi_h)={}&\mass{Y_h}{v_h}{\phi_h} + \stiff{Y_h}{Y_h}{\phi_h}\\
 &+ h^2\stiff{Y_h}{v_h}{\phi_h}.
 \end{aligned}
\end{equation}
\begin{lemma}\label{lem:general-defect-expansion}
Assume that $Y,Y_t,Y_{tt},Y_{ttt}$ are continuous in time with values in $[W^{6,\infty}(I)]^2$, $[W^{6,\infty}(I)]^2$, $[W^{3, \infty}(I)]^2$, and $[W^{2, \infty}(I)]^2$, respectively, and that $\alpha_0[Y], \beta_0[Y],\alpha_2[Y],\beta_2[Y] \in C([0,T];H^3(I))\cap C^1([0,T];H^2(I))$. Suppose that
\begin{equation}\label{eq:general-defect-regularity}
 \begin{aligned}
 &\inf_{(\rho,t)\in I\times[0,T]}|Y_\rho(\rho,t)|\geq c>0,\\
 &\sup_{0\leq t\leq T}\Bigl( \norm{Y(t)}{W^{6,\infty}(I)} +\norm{Y_t(t)}{W^{6,\infty}(I)} +\norm{Y_{tt}(t)}{W^{3,\infty}(I)}+\norm{Y_{ttt}(t)}{W^{2,\infty}(I)}\\
 &\qquad \qquad +\norm{V_2[Y(t)]}{W^{3,\infty}(I)}\Bigr)\leq M,
 \end{aligned}
\end{equation}
There exist $h_0=h_0(c^{-1},M)>0$ and $C=C(c^{-1},M)$ such that, for $0<h\leq h_0$ and $0\leq t\leq T-h^2$, the following estimates hold.

For every $\phi_h\in[S_h]^2$,
\begin{equation}\label{eq:defect-expansion}
 \begin{aligned}
 &\Bigg|d_h(\phi_h) -\int_I|Y_\rho|(\alpha_0+\kappa)\vec n\cdot\phi_h\,\dd\rho\\
 &\quad-h^2\int_I\Bigg[ |Y_\rho|\left(\alpha_2+\frac12(V_0)_t\cdot\vec n\right)\vec n\\
 &\qquad\quad -\frac1{12}(|Y_\rho|\alpha_0\vec n)_{\rho\rho} +\frac16(|Y_\rho|\vec n\otimes\vec n)_{\rho\rho}V_0 +\frac{|Y_\rho|^3\kappa^2}{24} (9\alpha_0\vec n-8V_0)\\
 &\qquad\quad -\left(\partial_sV_0 +\frac{|Y_\rho|^2\kappa^2}{24}\vec\tau -\frac{(|Y_\rho|)_\rho\kappa}{12}\vec n\right)_\rho \Bigg]\cdot\phi_h\,\dd\rho\Bigg| \leq Ch^4\norm{\phi_h}{0,h}.
 \end{aligned}
\end{equation}
For every $\beta_h\in S_h$,
\begin{equation}\label{eq:general-tangential-defect-expansion}
 \begin{aligned}
 &\Bigg|d_h\bigl(\Ih[\beta_hT_h[Y_h]]\bigr)-h^2\int_I|Y_\rho|\Bigl[ (-\partial_{ss}+\kappa^2)\beta_0 -2\kappa\partial_s\alpha_0\\
 &\qquad -\alpha_0\partial_s\kappa -\frac14(|Y_\rho|)_\rho\kappa^2 \Bigr]\beta_h\,\dd\rho\\
 &\quad-h^4\int_I\Bigg[ S_4[Y]+S_2[V_0,Y]+|Y_\rho|\Bigl[ (-\partial_{ss}+\kappa^2)\beta_2 -2\kappa\partial_s\alpha_2-\alpha_2\partial_s\kappa \Bigr]\\
 &\qquad -\frac12\left(\frac{(V_0)_{t\rho}}{|Y_\rho|}\right)_\rho \cdot\vec\tau \Bigg]\beta_h\,\dd\rho\Bigg| \leq Ch^6\norm{\beta_h}{0,h}.
 \end{aligned}
\end{equation}
\end{lemma}

\begin{proof}
Taylor expansion gives
\begin{equation}\label{eq:time-increment-estimates}
 \begin{aligned}
 &\norm{v}{W^{6,\infty}(I)} \leq \sup_{0\leq t\leq T}\norm{Y_t}{W^{6,\infty}(I)},\\
 &\norm{v-V_0}{W^{3,\infty}(I)} \leq h^2\sup_{0\leq t\leq T}\norm{V_2[Y]}{W^{3,\infty}(I)} +\frac{h^2}{2} \sup_{0\leq t\leq T}\norm{Y_{tt}}{W^{3,\infty}(I)},\\
 &\left\|v-V_0-h^2\left(V_2+\frac12Y_{tt}\right)\right\|_{W^{2,\infty}(I)} \leq\frac{h^4}{6} \sup_{0\leq t\leq T}\norm{Y_{ttt}}{W^{2,\infty}(I)}.
 \end{aligned}
\end{equation}

For the defect \eqref{eq:defect-expansion}, apply Lemma~\ref{lem:weakquadrature, mass-lumping} to $u=v$ for the mass-lumping part gives
\begin{equation}\label{eq:general-mass-expansion}
 \begin{aligned}
 &\Bigg|\mass{Y_h}{v_h}{\phi_h} -\int_I |Y_\rho|\alpha_0\vec n\cdot\phi_h\,\dd\rho\\
 &\quad-h^2\int_I |Y_\rho|\left(\alpha_2 + \frac{1}{2}(V_0)_t\cdot\vec n\right)\vec n\cdot\phi_h\,\dd\rho\\
 &\quad-\frac{h^2}{12}\int_I\Bigg[-\bigl(|Y_\rho|\alpha_0\vec n\bigr)_{\rho\rho} + 2\bigl(|Y_\rho|\vec n\otimes\vec n\bigr)_{\rho\rho}V_0\\
 &\qquad\qquad +\frac{|Y_\rho|^3\kappa^2}{2} \bigl(9\alpha_0\vec n-8V_0\bigr) \Bigg]\cdot\phi_h\,\dd\rho\Bigg|\\
 &\leq C h^4 \sup_{0\leq t\leq T}\left(\norm{Y_{ttt}}{L^{2}(I)} + \left\|V_2[Y]\right\|_{H^2(I)} + \left\|Y_{tt}\right\|_{H^2(I)}\right)\norm{\phi_h}{0,h}\\
 &\quad + \Bigg|\mass{Y_h}{v_h}{\phi_h} -\int_I |Y_\rho|(v\cdot\vec n)\vec n\cdot\phi_h\,\dd\rho\\
 &\quad-\frac{h^2}{12}\int_I\Bigg[-\bigl(|Y_\rho|(v\cdot\vec n)\vec n\bigr)_{\rho\rho} + 2\bigl(|Y_\rho|\vec n\otimes\vec n\bigr)_{\rho\rho}v\\
 &\qquad\qquad +\frac{|Y_\rho|^3\kappa^2}{2} \bigl(9(v\cdot\vec n)\vec n-8v\bigr) \Bigg]\cdot\phi_h\,\dd\rho\Bigg|\\
 &\leq C h^4 \norm{\phi_h}{0,h}.
 \end{aligned}
\end{equation}

For the stiffness terms, use Lemma~\ref{lem:weakflux} to $u=v$, $u=Y$ respectively gives
\begin{equation}\label{eq:general-combined-stiffness}
 \begin{aligned}
 &\Bigg| \stiff{Y_h}{Y_h}{\phi_h} +h^2\stiff{Y_h}{v_h}{\phi_h} -\int_I|Y_\rho|\kappa\vec n\cdot\phi_h\,\dd\rho\\
 &\quad+h^2\int_I\left( \partial_sV_0 +\frac{|Y_\rho|^2\kappa^2}{24}\vec\tau -\frac{(|Y_\rho|)_\rho\kappa}{12}\vec n \right)_\rho\cdot\phi_h\,\dd\rho\Bigg| \leq Ch^4\norm{\phi_h}{0,h}.
 \end{aligned}
\end{equation}
Combining \eqref{eq:general-mass-expansion} and \eqref{eq:general-combined-stiffness} proves \eqref{eq:defect-expansion}.

For the tangential defect, nodal orthogonality gives exactly
\begin{equation}\label{eq:general-tangential-mass-zero}
 \mass{Y_h}{v_h}{\Ih[\beta_hT_h[Y_h]]}=0.
\end{equation}
Lemmas~\ref{lem:tangentexpansion} and \ref{lem:tangentflux} therefore give
\begin{equation}\label{eq:assembled-tangential-expansion}
 \begin{aligned}
 &\Bigg|d_h\bigl(\Ih[\beta_hT_h[Y_h]]\bigr)-h^2\int_I\left[ -\frac14|Y_\rho|(|Y_\rho|)_\rho\kappa^2 -(\partial_sv)_\rho\cdot\vec\tau \right]\beta_h\,\dd\rho\\
 &\quad-h^4\int_I\bigl(S_4[Y]+S_2[v,Y]\bigr)\beta_h\,\dd\rho \Bigg| \leq Ch^6\norm{\beta_h}{0,h}.
 \end{aligned}
\end{equation}
Since $(\partial_sv)_\rho$ and $S_2[v, Y]$ are linear in $v$, we have
\begin{equation}\label{eq:general-tangential-time-errors}
 \begin{aligned}
 &\left\| (\partial_sv)_\rho-(\partial_sV_0)_\rho -h^2\left(\partial_s\left(V_2+\frac12Y_{tt}\right)\right)_\rho \right\|_{L^\infty(I)}\leq Ch^4,\\
 &\norm{S_2[v,Y]-S_2[V_0,Y]}{L^\infty(I)} \leq C\norm{v-V_0}{W^{3,\infty}(I)}\leq Ch^2.
 \end{aligned}
\end{equation}
For arbitrary  $\alpha,\beta\in H^2(I)$, the Frenet identities yield
\begin{equation}\label{eq:general-tangential-leading-identity}
 -\bigl(\partial_s(\alpha\vec n+\beta\vec\tau)\bigr)_\rho \cdot\vec\tau =|Y_\rho|\bigl[ (-\partial_{ss}+\kappa^2)\beta -2\kappa\partial_s\alpha-\alpha\partial_s\kappa\bigr].
\end{equation}
Finally, substitute \eqref{eq:general-tangential-time-errors} into \eqref{eq:assembled-tangential-expansion}, and use \eqref{eq:general-tangential-leading-identity} with $(\alpha,\beta)=(\alpha_0,\beta_0)$ and $(\alpha_2,\beta_2)$. This proves \eqref{eq:general-tangential-defect-expansion}.
\end{proof}

With the defect expansion, we can prove the comparison flows and improved defects.

\begin{theorem}[Approximation flows and improved defects]\label{thm:comparison-flow-defects}
For a regular curve $Y$, set $\alpha_0[Y]=-\kappa$ and let $\beta_0[Y]$ be the  solution of
\begin{equation}\label{eq:comparison-beta0}
 (-\partial_{ss}+\kappa^2)\beta_0 =\frac14(|Y_\rho|)_\rho\kappa^2 -3\kappa\partial_s\kappa.
\end{equation}
Set $V_0[Y] = \alpha_0[Y]\vec n + \beta_0[Y]\vec\tau$, define
\begin{equation}\label{eq:comparison-alpha2}
 \begin{aligned}
 \alpha_2[Y]={}& -\frac12\partial_{ss}\kappa+\frac32\kappa^3 -2\kappa\partial_s\beta_0+\frac12\kappa\beta_0^2 -\frac1{12}\kappa_{\rho\rho} -\frac{(|Y_\rho|)_\rho}{4|Y_\rho|}\kappa_\rho -\frac14|Y_\rho|^2\kappa^3\\
 &-\left(\frac12(|Y_\rho|)_\rho\kappa +\frac16|Y_\rho|\kappa_\rho\right)\beta_0,
 \end{aligned}
\end{equation}
and let $\beta_2[Y]$ be the solution of
\begin{equation}\label{eq:comparison-beta2}
 \begin{aligned}
 (-\partial_{ss}+\kappa^2)\beta_2 ={}&2\kappa\partial_s\alpha_2+\alpha_2\partial_s\kappa -\frac{S_4[Y]+S_2[V_0,Y]}{|Y_\rho|}\\
 &-\frac{\kappa^2}{8|Y_\rho|^3} \left[|Y_\rho|^4 (\partial_s\beta_0-\kappa^2)\right]_\rho\\
 &-\frac{(|Y_\rho|)_\rho\kappa}{4} \left(\partial_{ss}\kappa+\kappa^3 +\beta_0\partial_s\kappa\right) -\frac14\partial_s \left[(\partial_s\kappa+\kappa\beta_0)^2\right].
 \end{aligned}
\end{equation}
All quantities in \eqref{eq:comparison-beta0}--\eqref{eq:comparison-beta2} are evaluated on the same curve $Y$.

Let $Y^{[0]}$ and $Y^{[2]}$ be the solution of
\begin{equation}\label{eq:comparison-flows}
 \begin{aligned}
 Y_t^{[0]}={}& -\kappa[Y^{[0]}]\vec n[Y^{[0]}] +\beta_0[Y^{[0]}]\vec\tau[Y^{[0]}],\\
 Y_t^{[2]}={}& \bigl(-\kappa[Y^{[2]}]+h^2\alpha_2[Y^{[2]}]\bigr) \vec n[Y^{[2]}] +\bigl(\beta_0[Y^{[2]}]+h^2\beta_2[Y^{[2]}]\bigr) \vec\tau[Y^{[2]}],
 \end{aligned}
\end{equation}
with a prescribed initial parametrization $Y_0$. For $r\in\{0,2\}$, set $Y_h^{[r]}=\Ih Y^{[r]}$ and define $d_h^{[r]}$ by \eqref{eq:assembled-defect-definition} with $Y=Y^{[r]}$.

Assume that both solutions satisfy the regularity assumptions of Lemma~\ref{lem:general-defect-expansion} for $t\in [0, T]$. Then there exist $h_0>0$ and $C>0$ such that, for $0<h\leq h_0$, $t\in [0, T-h^2]$, $\phi_h\in[S_h]^2$, $\beta_h\in S_h$, the following estimates hold.
\begin{equation}\label{eq:zeroth-comparison-defects}
 \begin{aligned}
 |d_h^{[0]}(\phi_h)|\leq Ch^2\norm{\phi_h}{0,h},\qquad \left|d_h^{[0]}\bigl( \Ih[\beta_hT_h[Y_h^{[0]}]]\bigr)\right| \leq Ch^4\norm{\beta_h}{0,h},
 \end{aligned}
\end{equation}
and
\begin{equation}\label{eq:second-comparison-defects}
 \begin{aligned}
 |d_h^{[2]}(\phi_h)| \leq Ch^4\norm{\phi_h}{0,h}, \qquad \left|d_h^{[2]}\bigl( \Ih[\beta_hT_h[Y_h^{[2]}]]\bigr)\right| \leq Ch^6\norm{\beta_h}{0,h}.
 \end{aligned}
\end{equation}
\end{theorem}

\begin{proof}
First, the operator $-\partial_{ss}+\kappa^2$ is invertible, thus $\beta_0[Y]$ and $\beta_2[Y]$ given by \eqref{eq:comparison-beta0} and \eqref{eq:comparison-beta2} are well-defined.

For $Y^{[0]}$, the $O(1)$ term in \eqref{eq:defect-expansion} vanishes since $\alpha_0=-\kappa$. $\beta_0$ given by \eqref{eq:comparison-beta0} also cancels the $O(h^2)$ term in \eqref{eq:general-tangential-defect-expansion}. This proves \eqref{eq:zeroth-comparison-defects}.

For \eqref{eq:second-comparison-defects}, we write $Y=Y^{[2]}$ for short. From \eqref{eq:Frenet-relations} we have
\begin{equation}\label{eq:alpha2-frame-derivatives}
 \begin{aligned}
 \vec n_\rho&=|Y_\rho|\kappa\vec\tau, &\vec\tau_\rho&=-|Y_\rho|\kappa\vec n,\\
 \vec n_{\rho\rho} &=(|Y_\rho|\kappa)_\rho\vec\tau -|Y_\rho|^2\kappa^2\vec n, &\vec\tau_{\rho\rho} &=-(|Y_\rho|\kappa)_\rho\vec n -|Y_\rho|^2\kappa^2\vec\tau.
 \end{aligned}
\end{equation}
Moreover,
\begin{equation}\label{eq:alpha2-velocity-stiffness}
 \begin{aligned}
 \partial_sV_0 &=-(\partial_s\kappa+\kappa\beta_0)\vec n +(\partial_s\beta_0-\kappa^2)\vec\tau,\\
 -(\partial_sV_0)_\rho\cdot\vec n &=|Y_\rho|\left(\partial_{ss}\kappa +\beta_0\partial_s\kappa +2\kappa\partial_s\beta_0-\kappa^3\right),\\
 -(\partial_sV_0)_\rho\cdot\vec\tau &=|Y_\rho|\left((-\partial_{ss}+\kappa^2)\beta_0 +3\kappa\partial_s\kappa\right).
 \end{aligned}
\end{equation}
Thus the tangential component of the $O(h^2)$ coefficient in \eqref{eq:defect-expansion} is
\begin{equation}\label{eq:alpha2-tangential-coefficient}
 \begin{aligned}
 &\Bigg[ |Y_\rho|\left(\alpha_2+\frac12(V_0)_t\cdot\vec n\right)\vec n -\frac1{12}(|Y_\rho|\alpha_0\vec n)_{\rho\rho} +\frac16(|Y_\rho|\vec n\otimes\vec n)_{\rho\rho}V_0\\
 &\qquad +\frac{|Y_\rho|^3\kappa^2}{24}(9\alpha_0\vec n-8V_0) -\left(\partial_sV_0+\frac{|Y_\rho|^2\kappa^2}{24}\vec\tau -\frac{(|Y_\rho|)_\rho\kappa}{12}\vec n\right)_\rho \Bigg]\cdot\vec \tau\\
 &=\frac14|Y_\rho|(|Y_\rho|)_\rho\kappa^2 +\frac14|Y_\rho|^2\kappa\kappa_\rho\\
 &\quad +\frac16\left( -3|Y_\rho|(|Y_\rho|)_\rho\kappa^2 -|Y_\rho|^2\kappa\kappa_\rho +2|Y_\rho|^3\kappa^2\beta_0 \right) -\frac13|Y_\rho|^3\kappa^2\beta_0\\
 &\quad +|Y_\rho|\left[ (-\partial_{ss}+\kappa^2)\beta_0 +3\kappa\partial_s\kappa \right] -\frac1{24}\bigl(|Y_\rho|^2\kappa^2\bigr)_\rho +\frac1{12}|Y_\rho|(|Y_\rho|)_\rho\kappa^2\\
 &=|Y_\rho|\left[ (-\partial_{ss}+\kappa^2)\beta_0 +3\kappa\partial_s\kappa -\frac14(|Y_\rho|)_\rho\kappa^2 \right] = 0.
 \end{aligned}
\end{equation}

Using the evolution equation of $\kappa$ and $\vec n$ in \cite[Lemma 37, Lemma 39]{BGN2020} gives
\begin{equation}\label{eq:corrected-flow-geometry}
 \begin{aligned}
 \kappa_t & =\partial_{ss}\kappa+\kappa^3+\beta_0\partial_s\kappa +h^2(-\partial_{ss}\alpha_2-\kappa^2\alpha_2 +\beta_2\partial_s\kappa),\\
 \vec{n}_t & =\left(\partial_s\kappa+\kappa\beta_0 +h^2(-\partial_s\alpha_2+\kappa\beta_2)\right)\vec\tau.\\
 \vec\tau_t & =\left(-\partial_s\kappa-\kappa\beta_0 +h^2(\partial_s\alpha_2-\kappa\beta_2)\right)\vec n.
 \end{aligned}
\end{equation}
Therefore,
\begin{equation}\label{eq:alpha2-normal-acceleration}
 \begin{aligned}
 (V_0)_t\cdot\vec n &=-\partial_{ss}\kappa-\kappa^3 -2\beta_0\partial_s\kappa-\kappa\beta_0^2\\
 &\quad+h^2\left[ \partial_{ss}\alpha_2+\kappa^2\alpha_2 +\beta_0\partial_s\alpha_2 -\beta_2(\partial_s\kappa+\kappa\beta_0)\right].
 \end{aligned}
\end{equation}

Substituting \eqref{eq:alpha2-velocity-stiffness}, \eqref{eq:alpha2-normal-acceleration}, and $\alpha_2$ given by \eqref{eq:comparison-alpha2} into \eqref{eq:defect-expansion}, one could verify that the normal component of the $O(h^2)$ coefficient vanishes. The remaining term is of the order $O(h^2)$, which is $\frac{h^2|Y_\rho|}{2}\left( \partial_{ss}\alpha_2+\kappa^2\alpha_2 +\beta_0\partial_s\alpha_2 -\beta_2(\partial_s\kappa+\kappa\beta_0)\right)$ from $(V_0)_t \cdot \vec{n}$. Under the regularity assumptions of the theorem,
\begin{equation}\label{eq:alpha2-normal-remainder-bound}
 \begin{aligned}
 &\left\| \frac{h^2|Y_\rho|}{2}\left[ \partial_{ss}\alpha_2+\kappa^2\alpha_2 +\beta_0\partial_s\alpha_2 -\beta_2(\partial_s\kappa+\kappa\beta_0)\right] \right\|_{L^2(I)}\\
 &\leq Ch^2\left( \norm{\alpha_2}{H^2(I)}+\norm{\beta_2}{L^2(I)}\right) \leq Ch^2,
 \end{aligned}
\end{equation}
which proves $|d_h^{[2]}(\phi_h)|\leq Ch^4\norm{\phi_h}{0,h}$.

To treat the $O(h^4)$ in the tangential defect \eqref{eq:general-tangential-defect-expansion}, the only unknown term is $\left(\frac{(V_0)_{t\rho}}{|Y_\rho|}\right)_\rho \cdot \vec{\tau} = (\partial_s \partial_t(V_0))_\rho \cdot \vec{\tau}$. Using commutator $[\partial_t,\partial_s] =-(|Y_\rho|)_t|Y_\rho|^{-1}\partial_s$, together with \eqref{eq:alpha2-velocity-stiffness} gives
\begin{equation}\label{eq:comparison-acceleration-identity}
 \begin{aligned}
 \left(\frac{(V_0)_{t\rho}}{|Y_\rho|}\right)_\rho\cdot\vec\tau ={}&-\frac14 \left(|Y_\rho|(|Y_\rho|)_\rho\kappa^2\right)_t+\left(\frac{(|Y_\rho|)_t}{|Y_\rho|} \partial_sV_0\right)_\rho\cdot\vec\tau -(\partial_sV_0)_\rho\cdot\vec\tau_t.
 \end{aligned}
\end{equation}
The evolution equation \eqref{eq:comparison-flows} gives
\begin{equation}\label{eq:beta2-metric-evolution}
 \frac{(|Y_\rho|)_t}{|Y_\rho|} =\partial_sY_t\cdot\vec\tau =\partial_s\beta_0-\kappa^2 +h^2(\kappa\alpha_2+\partial_s\beta_2).
\end{equation}

Substituting \eqref{eq:beta2-metric-evolution} into \eqref{eq:comparison-acceleration-identity} gives the exact identity
\begin{equation}\label{eq:beta2-explicit-acceleration-remainder}
 \begin{aligned}
 &\frac1{|Y_\rho|} \left(\frac{(V_0)_{t\rho}}{|Y_\rho|}\right)_\rho\cdot\vec\tau +\frac{\kappa^2}{4|Y_\rho|^3} \left[|Y_\rho|^4(\partial_s\beta_0-\kappa^2)\right]_\rho\\
 &\quad+\frac{(|Y_\rho|)_\rho\kappa}{2} \left(\partial_{ss}\kappa+\kappa^3 +\beta_0\partial_s\kappa\right) +\frac12\partial_s \left[(\partial_s\kappa+\kappa\beta_0)^2\right]\\
 &=h^2\left( -\frac34(|Y_\rho|)_\rho\kappa^2 (\kappa\alpha_2+\partial_s\beta_2)+\left(\frac{\partial_s\beta_0-\kappa^2}{|Y_\rho|} -\frac{|Y_\rho|\kappa^2}{4}\right) (\kappa\alpha_2+\partial_s\beta_2)_\rho\right.\\
 &\quad \left. +\frac{(|Y_\rho|)_\rho\kappa}{2} \left(\partial_{ss}\alpha_2+\kappa^2\alpha_2 -\beta_2\partial_s\kappa\right)\right.\\
 &\quad \left. +(\partial_s\alpha_2-\kappa\beta_2) \left(\partial_{ss}\kappa+\beta_0\partial_s\kappa +2\kappa\partial_s\beta_0-\kappa^3\right) \right).
 \end{aligned}
\end{equation}
Since $\alpha_2, \beta_0, \beta_2 \in H^2(I)$, the $L^2(I)$ norm of the RHS is bounded by $Ch^2$. Applying \eqref{eq:general-tangential-defect-expansion} and \eqref{eq:comparison-beta2} to \eqref{eq:beta2-explicit-acceleration-remainder} therefore gives
\begin{equation}\label{eq:beta2-final-tangential-defect}
 \left|d_h^{[2]}\bigl(\Ih[\beta_hT_h[Y_h^{[2]}]]\bigr)\right| \leq Ch^4h^2\norm{\beta_h}{0,h} +Ch^6\norm{\beta_h}{0,h} \leq Ch^6\norm{\beta_h}{0,h}.
\end{equation}
This proves the second estimate in \eqref{eq:second-comparison-defects}.
\end{proof}

\begin{remark}\label{rem:lambda}
For $\tau=\lambda h^2$ with a fixed $\lambda>0$, the tangential velocity $\beta_0$ satisfies
\begin{equation}\label{eq:comparison-beta0-lambda}
(-\partial_{ss}+\kappa^2)\beta_0 =\frac{1}{4\lambda}(|Y_\rho|)_\rho\kappa^2 -3\kappa\partial_s\kappa .
\end{equation}
Formally letting $\lambda\to0$, which corresponds to the semi-discrete BGN method,
\begin{equation*}
    \frac{1}{4}(|Y_\rho|)_\rho\kappa^2=0.
\end{equation*}
Hence, either $\kappa=0$ or $\partial_\rho|Y_\rho|=0$. This is the well-known equipartition property of the BGN method \cite[Theorem 79, (ii)]{BGN2020}: any two neighbouring edges are either parallel $(\kappa=0)$ or of equal length $(\partial_\rho|Y_\rho|=0)$.
\end{remark}

\subsection{Convergence of the approximation flows}
\begin{assumption}\label{ass:existence}
Suppose that $Y^{[0]}(0)= Y^{[2]}(0)$ and the classical solutions for $Y^{[0]}(t)$, $Y^{[2]}(t)$ exist on a common interval $[0,T]$, and belong to $C([0,T];W^{3,\infty}(I;\mathbb R^2)) \cap C^1([0,T];H^1(I;\mathbb R^2))$. Assume that, independently of $h$,
\begin{equation}\label{eq:continuous-regularity}
 \begin{aligned}
 &\inf_{I\times[0,T]}|Y^{[0]}_\rho|\geq c_0>0, \inf_{I\times[0,T]}|Y^{[2]}_\rho|\geq c_0>0,\\
 &\sup_{0\leq t\leq T} \bigl(\|Y^{[2]}(t)\|_{W^{3,\infty}(I)} +\|Y^{[0]}(t)\|_{W^{3,\infty}(I)} +\|Y^{[0]}_t(t)\|_{W^{2,\infty}(I)}\\
 &\qquad + \|V_2[Y^{[2]}]\|_{L^2(0,T;H^1(I))}\bigr)\leq M.
 \end{aligned}
\end{equation}
\end{assumption}

\begin{lemma}\label{lem:continuous-stability}
Set $e_x=Y^{[2]}-Y^{[0]}$ and $e_v=Y^{[2]}_t-Y^{[0]}_t$. Under Assumption \ref{ass:existence}, there are $h_0>0$ and $C>0$, depending only on $c_0,M,T$ such that, for $h\leq h_0$,
\begin{equation}\label{eq:continuous-stability-result}
 \begin{aligned}
 &\sup_{0\leq t\leq T}\|e_x(t)\|_{H^1(I)}\leq Ch^2, \qquad \sup_{0\leq t\leq T}\|e_x(t)\|_{W^{1,\infty}(I)} \leq h^1.
 \end{aligned}
\end{equation}
\end{lemma}

\begin{proof}
For notational convenience, we take $X = Y^{[2]}$ and $Y = Y^{[0]}$. We first work on an interval $[0, t^*]$ on which $\|e_x\|_{W^{1,\infty}(I)}\leq h^1$.

\textit{Step 1: normal error.} For any $\phi\in H^1$, \eqref{eq:comparison-flows} gives
\begin{equation}\label{eq:continuous-normal-variational}
 \begin{aligned}
 &\int_{\Gamma[X]}(X_t\cdot\vec n[X]) (\phi\cdot\vec n[X]) +\int_{\Gamma[X]}\partial_{s[X]}X\cdot \partial_{s[X]}\phi =h^2\int_{\Gamma[X]}\alpha_2[X](\phi\cdot\vec n[X]),\\
 &\int_{\Gamma[Y]}(Y_t\cdot\vec n[Y]) (\phi\cdot\vec n[Y]) +\int_{\Gamma[Y]}\partial_{s[Y]}Y\cdot \partial_{s[Y]}\phi=0.
 \end{aligned}
\end{equation}
Testing with $e_v$ and subtracting yields
\begin{equation}\label{eq:continuous-normal-test}
 \begin{aligned}
 &\int_{\Gamma[X]}(e_v\cdot\vec n[X])^2 +\int_I\frac{(e_x)_\rho\cdot\vec n[Y]}{|Y_\rho|}\vec n[Y]\cdot(e_v)_\rho\,\mathrm d\rho\\
 &=h^2\int_{\Gamma[X]}\alpha_2[X](e_v\cdot\vec n[X])-\int_I\Bigl( |X_\rho|(Y_t\cdot\vec n[X])\vec n[X]\\
 &\qquad -|Y_\rho|(Y_t\cdot\vec n[Y])\vec n[Y] \Bigr)\cdot e_v\,\mathrm d\rho \\
 & \qquad + \int_{I}\left(\frac{|Y_\rho|-|X_\rho|}{|Y_\rho||X_\rho|} \bigl((e_x)_\rho\cdot\vec n[Y]\bigr)\vec n[Y] -\frac12|\vec\tau[X]-\vec\tau[Y]|^2\vec\tau[Y]\right) \cdot (e_v)_\rho \, \mathrm d\rho.
 \end{aligned}
\end{equation}
where Lemma \ref{lem:normalforce} is used to simplify $\int_{\Gamma[X]}\partial_{s[X]}X\cdot \partial_{s[X]}\phi - \int_{\Gamma[Y]}\partial_{s[Y]}Y\cdot \partial_{s[Y]}\phi$.

The right-hand side is bounded by $C h^2  \left\|e_v\right\|_{L^2} + C \left\|e_x\right\|_{H^1} \left\|e_v\right\|_{L^2}$. For the left-hand side, the derivative trick gives
\begin{equation}\label{eq:continuous-normal-derivative}
 \begin{aligned}
 &\int_{\Gamma[Y]}(\partial_{s[Y]}e_x\cdot\vec n[Y]) (\partial_{s[Y]}e_v\cdot\vec n[Y]) =\frac12\frac{\mathrm d}{\mathrm dt} \left\|\partial_{s[Y]}e_x\cdot\vec n[Y] \right\|_{L^2(\Gamma[Y])}^2\\
 &\quad-\int_I \frac{\bigl((e_x)_\rho\cdot\vec n[Y]\bigr) \bigl((e_x)_\rho\cdot\partial_t\vec n[Y]\bigr)} {|Y_\rho|}\,\mathrm d\rho +\frac12\int_I\frac{\partial_t|Y_\rho|}{|Y_\rho|^2} |(e_x)_\rho\cdot\vec n[Y]|^2\,\mathrm d\rho\\
 &\geq\frac12\frac{\mathrm d}{\mathrm dt} \left\|\partial_{s[Y]}e_x\cdot\vec n[Y] \right\|_{L^2(\Gamma[Y])}^2-C\|e_x\|_{H^1(I)}^2.
 \end{aligned}
\end{equation}
Combining \eqref{eq:continuous-normal-test}-- \eqref{eq:continuous-normal-derivative}, we obtain
\begin{equation}\label{eq:continuous-normal-estimate}
 \begin{aligned}
 &\frac12\frac{\mathrm d}{\mathrm dt} \left\|\partial_{s[Y]}e_x\cdot\vec n[Y] \right\|_{L^2(\Gamma[Y])}^2 +c\|e_v\cdot\vec n[X]\|_{L^2}^2\\
 &\qquad \leq C h^2  \left\|e_v\right\|_{L^2} +C\left\|e_x\right\|_{H^1}\left\|e_v\right\|_{L^2}+C\left\|e_x\right\|_{H^1}^2.
 \end{aligned}
\end{equation}

\textit{Step 2: tangential error.} For $Z=X,Y$ and any $\beta\in H^1(I)$, direct computation gives
\begin{equation}\label{eq:continuous-tangential-constraint}
 \begin{aligned}
 &\int_{\Gamma[Z]}\partial_{s[Z]}V_0[Z]\cdot \partial_{s[Z]}(\beta\vec\tau[Z])\\
 &=\int_{\Gamma[Z]}\Bigl( \beta H[Z]\partial_{s[Z]}H[Z] +\beta\beta_0[Z]H[Z]^2\\
 &\qquad +(\partial_{s[Z]}\beta_0[Z])(\partial_{s[Z]}\beta) -H[Z]^2\partial_{s[Z]}\beta \Bigr)\\
 &=\int_{\Gamma[Z]}\beta\Bigl( (-\partial_{s[Z]}^2+H[Z]^2)\beta_0[Z] +3H[Z]\partial_{s[Z]}H[Z]\Bigr)\\
 &=\frac14\int_I \frac{\partial_\rho|Z_\rho|}{|Z_\rho|} |\partial_\rho\vec\tau[Z]|^2\beta\,\mathrm d\rho.
 \end{aligned}
\end{equation}
The velocity equations give, for any $\phi\in H^1$,
\begin{equation}\label{eq:continuous-velocity-variational}
 \begin{aligned}
 &\int_{\Gamma[X]}\partial_{s[X]}X_t\cdot \partial_{s[X]}\phi =\int_{\Gamma[X]}\partial_{s[X]}V_0[X]\cdot \partial_{s[X]}\phi\\
 &\qquad +h^2\int_{\Gamma[X]}\partial_{s[X]}V_2[X]\cdot \partial_{s[X]}\phi,\\
 &\int_{\Gamma[Y]}\partial_{s[Y]}Y_t\cdot \partial_{s[Y]}\phi =\int_{\Gamma[Y]}\partial_{s[Y]}V_0[Y]\cdot \partial_{s[Y]}\phi.
 \end{aligned}
\end{equation}
Take $\phi=\beta\vec\tau[X]$ and $\phi=\beta\vec\tau[Y]$, respectively. Subtracting and using \eqref{eq:continuous-tangential-constraint} gives
\begin{equation}\label{eq:continuous-paired-tests}
 \begin{aligned}
 &\int_{\Gamma[X]}\partial_{s[X]}e_v\cdot \partial_{s[X]}(\beta\vec\tau[X]) =h^2\int_{\Gamma[X]}\partial_{s[X]}V_2[X]\cdot \partial_{s[X]}(\beta\vec\tau[X])\\
 &\quad+\frac14\int_I\beta\left( \frac{\partial_\rho|X_\rho|}{|X_\rho|} |\partial_\rho\vec\tau[X]|^2 -\frac{\partial_\rho|Y_\rho|}{|Y_\rho|} |\partial_\rho\vec\tau[Y]|^2\right)\,\mathrm d\rho\\
 &\quad-\left( \int_{\Gamma[X]}\partial_{s[X]}Y_t\cdot \partial_{s[X]}(\beta\vec\tau[X]) -\int_{\Gamma[Y]}\partial_{s[Y]}Y_t\cdot \partial_{s[Y]}(\beta\vec\tau[Y]) \right)\\
 &:=h^2\int_{\Gamma[X]}\partial_{s[X]}V_2[X]\cdot \partial_{s[X]}(\beta\vec\tau[X]) +J_{21}(\beta)+J_{22}(\beta).
 \end{aligned}
\end{equation}
For $J_{21}$,  integration by parts gives
\begin{equation}\label{eq:continuous-j21-splitting}
 \begin{aligned}
 J_{21}(\beta) &=-\frac{1}{4}\int_I(\vec\tau[X]-\vec\tau[Y])\cdot \partial_\rho\left[ \beta\frac{\partial_\rho|X_\rho|}{|X_\rho|} (\partial_\rho\vec\tau[X]+\partial_\rho\vec\tau[Y]) \right]\,\mathrm d\rho\\
 &\quad-\frac{1}{4}\int_I \frac{|X_\rho|-|Y_\rho|}{|Y_\rho|} \partial_\rho\left( \beta\frac{|Y_\rho|}{|X_\rho|} |\partial_\rho\vec\tau[Y]|^2 \right)\,\mathrm d\rho.
 \end{aligned}
\end{equation}
Similarly,
\begin{equation}\label{eq:continuous-j22-splitting}
 \begin{aligned}
 J_{22}(\beta) &=-\int_I\left(\frac1{|X_\rho|}-\frac1{|Y_\rho|}\right) (Y_t)_\rho\cdot(\beta\vec\tau[X])_\rho\,\mathrm d\rho\\
 &\quad+\int_I\partial_\rho\left( \frac{(Y_t)_\rho}{|Y_\rho|}\right) \cdot\beta(\vec\tau[X]-\vec\tau[Y])\,\mathrm d\rho.
 \end{aligned}
\end{equation}
Hence
\begin{equation}\label{eq:continuous-tangential-rhs}
 \begin{aligned}
 &h^2\left|\int_{\Gamma[X]}\partial_{s[X]}V_2[X]\cdot \partial_{s[X]}(\beta\vec\tau[X])\right| \leq C h^2\left\|\beta\right\|_{H^1},\\
 &|J_{21}(\beta)|+|J_{22}(\beta)| \leq C\left\|e_x\right\|_{H^1}\left\|\beta\right\|_{H^1}.
 \end{aligned}
\end{equation}

Take $\beta=e_v\cdot\vec\tau[X]$, so that $\beta\vec\tau[X]=P^{\mathrm{tan}}[X]e_v$. For the left-hand side,  integration by parts and $P^N[X]e_v\cdot\vec\tau[X]=0$ give
\begin{equation}\label{eq:continuous-tangential-mixed}
 \begin{aligned}
 &\int_{\Gamma[X]}\partial_{s[X]}(P^N[X]e_v)\cdot \partial_{s[X]}(\beta\vec\tau[X])\\
 &=-\int_{\Gamma[X]}P^N[X]e_v\cdot \left(2(\partial_{s[X]}\beta)\partial_{s[X]}\vec\tau[X] +\beta\partial_{s[X]}^2\vec\tau[X]\right) \\
 &\geq-C\|e_v\cdot\vec n[X]\|_{L^2(I)} \|P^{\mathrm{tan}}[X]e_v\|_{H^1(I)}.
 \end{aligned}
\end{equation}
Consequently,
\begin{equation}\label{eq:continuous-tangential-lhs}
 \begin{aligned}
 &\int_{\Gamma[X]}\partial_{s[X]}e_v\cdot \partial_{s[X]}(P^{\mathrm{tan}}[X]e_v)\\
 &=\int_{\Gamma[X]} |\partial_{s[X]}(P^{\mathrm{tan}}[X]e_v)|^2 +\int_{\Gamma[X]}\partial_{s[X]}(P^N[X]e_v)\cdot \partial_{s[X]}(P^{\mathrm{tan}}[X]e_v)\\
 &\geq \int_{\Gamma[X]} |\partial_{s[X]}(P^{\mathrm{tan}}[X]e_v)|^2 -C\|e_v\cdot\vec n[X]\|_{L^2(I)} \|P^{\mathrm{tan}}[X]e_v\|_{H^1(I)}.
 \end{aligned}
\end{equation}

Follow the same steps in Lemma \ref{lem:tancontrol} to obtain
\begin{equation}\label{eq:continuous-tangent-poincare}
 \begin{aligned}
 \|P^{\mathrm{tan}}[X]e_v\|_{H^1}^2 \leq C \int_{\Gamma[X]} |\partial_{s[X]}(P^{\mathrm{tan}}[X]e_v)|^2.
 \end{aligned}
\end{equation}

Combining \eqref{eq:continuous-paired-tests}, \eqref{eq:continuous-tangential-rhs}, and \eqref{eq:continuous-tangential-lhs}, we obtain
\begin{equation}\label{eq:continuous-tangential-control}
 \|P^{\mathrm{tan}}[X]e_v\|_{H^1(I)} \leq C\bigl(h^2+\|e_x\|_{H^1(I)}+\|e_v\cdot\vec n[X]\|_{L^2(I)}\bigr).
\end{equation}

\textit{Step 3: recover the full norm.} From Lemma \ref{lem:norm}, we have
\begin{equation}\label{eq:continuous-norm-equivalence}
 \|e_x\|_{H^1(I)}^2\sim \left\|\partial_{s[Y]}e_x\cdot\vec n[Y] \right\|_{L^2(\Gamma[Y])}^2 +\|P^{\mathrm{tan}}[Y]e_x\|_{H^1(I)}^2+\|e_x\|_{L^2(I)}^2.
\end{equation}
Indeed, the identity $\partial_\rho(e_x\cdot\vec\tau[Y]) =(e_x)_\rho\cdot\vec\tau[Y]+e_x\cdot\partial_\rho\vec\tau[Y]$, together with the normal component of $(e_x)_\rho$, gives both inequalities. Since $(e_x)_t=e_v$,
\begin{equation}\label{eq:continuous-position-estimate}
 \frac12\frac{\mathrm d}{\mathrm dt}\|e_x\|_{L^2(I)}^2 \leq\|e_x\|_{L^2(I)}\|e_v\|_{L^2(I)}.
\end{equation}
For the tangential position error, use
\begin{equation}\label{eq:continuous-projection-derivative}
 \partial_t(P^{\mathrm{tan}}[Y]e_x) =P^{\mathrm{tan}}[X]e_v +(P^{\mathrm{tan}}[Y]-P^{\mathrm{tan}}[X])e_v +(\partial_tP^{\mathrm{tan}}[Y])e_x.
\end{equation}
We have $\|P^{\mathrm{tan}}[Y]e_x\|_{W^{2,\infty}(I)}\leq C$ and $\|P^{\mathrm{tan}}[Y]-P^{\mathrm{tan}}[X]\|_{L^2(I)}\leq C\|e_x\|_{H^1(I)}$. Together with $\|\partial_tP^{\mathrm{tan}}[Y]\|_{W^{1,\infty}(I)}\leq C$ from \eqref{eq:continuous-regularity}, we obtain
\begin{equation}\label{eq:continuous-tangential-position}
 \begin{aligned}
 &\frac12\frac{\mathrm d}{\mathrm dt} \|P^{\mathrm{tan}}[Y]e_x\|_{H^1(I)}^2 \leq C\|e_x\|_{H^1(I)}\|P^{\mathrm{tan}}[X]e_v\|_{H^1(I)}\\
 &\qquad +C\|e_x\|_{H^1(I)}\|e_v\|_{L^2(I)}+C\|e_x\|_{H^1(I)}^2.
 \end{aligned}
\end{equation}
Add \eqref{eq:continuous-normal-estimate}, \eqref{eq:continuous-tangential-control}, \eqref{eq:continuous-tangential-position}, we obtain
\begin{equation}\label{eq:continuous-energy-estimate}
 \begin{aligned}
 &\frac12\frac{\mathrm d}{\mathrm dt} \left( \left\|\partial_{s[Y]}e_x\cdot\vec n[Y] \right\|_{L^2(\Gamma[Y])}^2 +\|P^{\mathrm{tan}}[Y]e_x\|_{H^1(I)}^2+\|e_x\|_{L^2(I)}^2 \right)\\
 &\qquad+c\|e_v\cdot\vec n[X]\|_{L^2(I)}^2 +c\|P^{\mathrm{tan}}[X]e_v\|_{H^1(I)}^2\\
 &\leq C h^4 +C\left( \left\|\partial_{s[Y]}e_x\cdot\vec n[Y] \right\|_{L^2(\Gamma[Y])}^2 +\|P^{\mathrm{tan}}[Y]e_x\|_{H^1(I)}^2+\|e_x\|_{L^2(I)}^2 \right).
 \end{aligned}
\end{equation}
Since $e_x(0)=0$, Gronwall and the norm equivalence give
\begin{equation}\label{eq:continuous-gronwall}
 \begin{aligned}
 &\sup_{0\leq s\leq t}\|e_x(s)\|_{H^1(I)}^2 +\int_0^t\bigl(\|e_v\cdot\vec n[X]\|_{L^2(I)}^2 +\|P^{\mathrm{tan}}[X]e_v\|_{H^1(I)}^2\bigr)\leq Ch^4,
 \end{aligned}
\end{equation}
which gives the desired estimate.

Finally, continuous interpolation gives
\begin{equation}\label{eq:continuous-interpolation}
 \begin{aligned}
 \|e_x\|_{W^{1,\infty}(I)} &\leq C\left(\|e_x\|_{H^1(I)} +\|e_x\|_{H^1(I)}^{3/4}\|e_x\|_{H^3(I)}^{1/4}\right)\\
 &\leq C \left\|e_x\right\|_{H^1(I)} + C \left\|e_x\right\|_{H^1(I)}^{3/4}\left(\left\|X\right\|_{H^3(I)} + \left\|Y\right\|_{H^3(I)}^2\right)^{1/4} \leq Ch^{3/2}.
 \end{aligned}
\end{equation}
Since $1<3/2$, choosing $h_0$ sufficiently small makes $Ch^{3/2}\leq h^1$. Continuity therefore implies $t^*=T$, which proves \eqref{eq:continuous-stability-result}.
\end{proof}

\section{Proof of the main theorem}\label{sec:proof}
We first prove the following convergence between $X_h$ and the approximated curve $Y^{[0]}, Y^{[2]}$ in $L^{\infty} H^1$ norm.
\begin{theorem}[Convergence for $\tau=h^2$]\label{thm:main2}
Under Assumption~\ref{ass:regular}, let $Y^{[2]}$ as constructed in \eqref{eq:comparison-flows}. Fix $0<T\leq T_*$ and put $M_h=\lfloor T/\tau\rfloor$. Set
\begin{equation*}
 Y_h^m=\Ih Y^{[2]}(t_m=m \tau),\qquad e_x^m=X_h^m-Y_h^m,\qquad e_v^{m+1}=\frac{1}{\tau} (e_x^{m+1} - e_x^m).
\end{equation*}
There are $h_0>0$ and $C>0$, independent of $h$, such that the original BGN steps are well defined, the numerical polygons are embedded, and
\begin{equation}\label{eq:error bound h4}
 \max_{m}\norm{e_x^m}{1,h}\leq Ch^4,
\end{equation}
Consequently,
\begin{equation}\label{eq:error bound h2}
 \max_m\left(\norm{X_h^m-Y^{[0]}(t_m)}{L^2(I)} +h\norm{X_h^m-Y^{[0]}(t_m)}{H^1(I)}\right)\leq Ch^2.
\end{equation}
\end{theorem}

\begin{proof}The proof follows the same structure as in Lemma \ref{lem:continuous-stability}.

Define defect bounds $D_1^m,D_2^m$ by
\begin{equation*}
 |d_h^{m+1}(\phi_h)|\leq D_1^m\norm{\phi_h}{0,h},
\end{equation*}
\begin{equation*}
 |d_h^{m+1}(\pt{Y_h^m}\phi_h)| \leq D_2^m\norm{\pt{Y_h^m}\phi_h}{1,h}, \qquad \delta^m=D_1^m+\tau^{-1}D_2^m.
\end{equation*}
Theorem~\ref{thm:comparison-flow-defects} gives
\begin{equation*}
 D_1^m\leq Ch^4,\qquad D_2^m\leq Ch^6,\qquad \delta^m\leq Ch^4.
\end{equation*}

We prove by induction, first $\left\|e_x^0\right\|_{W^{1, \infty}} = 0 \leq h^3$. We assume that $\left\|e_x^k\right\|_{W^{1, \infty}}\leq h^3$ for $k = 1, \ldots, m$.

\textit{Step 1: normal error.} Subtracting the reference variational equation from the numerical, we obtain the error equation
\begin{equation}\label{eq:error eq}
 \begin{aligned}
 &\mass{X_h^m}{e_v^{m+1}}{\phi_h} +\tau\stiff{X_h^m}{e_v^{m+1}}{\phi_h}\\
 &\qquad+\stiff{X_h^m}{X_h^m}{\phi_h} -\stiff{Y_h^m}{Y_h^m}{\phi_h}\\
 &=-d_h^{m+1}(\phi_h) -\Bigl(\mass{X_h^m}{V_Y^{m+1}}{\phi_h}\\
 &\qquad\qquad -\mass{Y_h^m}{V_Y^{m+1}}{\phi_h}\Bigr)\\
 &\qquad-\tau\left(\stiff{X_h^m}{V_Y^{m+1}}{\phi_h} -\stiff{Y_h^m}{V_Y^{m+1}}{\phi_h}\right)\\
 &:= -d_h^{m+1}(\phi_h) + J_{11}(\phi_h) + \tau J_{12}(\phi_h).
 \end{aligned}
\end{equation}
Here $V_Y^{m+1}=\frac{1}{\tau} (Y_h^{m+1} - Y_h^m)$. The standard perturbation estimate gives
\begin{equation*}
 \begin{aligned}
 J_{11}(\phi_h)\leq C\norm{e_x^m}{1,h}\norm{\phi_h}{0,h}, \qquad J_{12}(\phi_h)\leq C\norm{e_x^m}{1,h}\norm{(\phi_h)_\rho}{L^2(I)}.
 \end{aligned}
\end{equation*}
Take $\phi_h = e_v^{m+1}$, the LHS of \eqref{eq:error eq} gives
\begin{equation}
 \begin{aligned}
 &\mass{X_h^m}{e_v^{m+1}}{\phi_h} +\tau\stiff{X_h^m}{e_v^{m+1}}{\phi_h}\\
 &\qquad+\stiff{X_h^m}{X_h^m}{\phi_h} -\stiff{Y_h^m}{Y_h^m}{\phi_h}\\
 &\geq c\left\|e_v^{m+1} \cdot \widetilde{N}_h[X_h^m]\right\|_{0, h}^2 + \tau \left\|\partial_{s[X_h^m]} e_v^{m+1}\right\|_{L^2(\Gamma[X_h^m])}^2\\
 & \quad - C \left\|e_x\right\|_{W^{1, \infty}(I)} \left\|e_x\right\|_{1, h}\left\|e_v\right\|_{1, h} \\
 & \quad +  \bigl(\partial_{s[Y_h^m]}e_x^m\cdot\vec n_h[Y_h^m], \partial_{s[Y_h^m]}e_v^{m+1}\cdot\vec n_h[Y_h^m]\bigr) _{\Gamma_h[Y_h^m]}
 \end{aligned}
\end{equation}
For the last term, use the derivative trick, we obtain
\begin{equation}
 \begin{aligned}
 &\bigl(\partial_{s[Y_h^m]}e_x^m\cdot\vec n_h[Y_h^m], \partial_{s[Y_h^m]}e_v^{m+1}\cdot\vec n_h[Y_h^m]\bigr) _{\Gamma_h[Y_h^m]} \\
 & = \frac{1}{2\tau}\left(\left\|\partial_{s[Y_h^m]}e_x^{m+1}\cdot\vec n_h[Y_h^m]\right\|_{L^2(\Gamma[Y_h^m])}^2 - \left\|\partial_{s[Y_h^m]}e_x^m\cdot\vec n_h[Y_h^m]\right\|_{L^2(\Gamma[Y_h^m])}^2\right) \\
 & \quad -\frac\tau2 \norm{\partial_{s[Y_h^m]}e_v^{m+1}\cdot \vec n_h[Y_h^m]}{L^2(\Gamma[Y_h^m])}^2 \\
 & = \frac{1}{2\tau}\Bigl(\left\|\partial_{s[Y_h^{m+1}]}e_x^{m+1}\cdot\vec n_h[Y_h^{m+1}]\right\|_{L^2(\Gamma[Y_h^{m+1}])}^2\\
 & \qquad\qquad - \left\|\partial_{s[Y_h^m]}e_x^m\cdot\vec n_h[Y_h^m]\right\|_{L^2(\Gamma[Y_h^m])}^2\Bigr) \\
 & \quad -\frac\tau2 \norm{\partial_{s[Y_h^m]}e_v^{m+1}\cdot \vec n_h[Y_h^m]}{L^2(\Gamma[Y_h^m])}^2 \\
 & \quad - \int_I \left( \frac1{|(Y_h^m)_\rho|} -\frac1{|(Y_h^{m+1})_\rho|} \right) |(e_x^{m+1})_\rho\cdot\vec n_h[Y_h^m]|^2\,\mathrm d\rho\\
 &\quad- \int_I \frac{ \bigl((e_x^{m+1})_\rho\cdot(\vec n_h[Y_h^m]-\vec n_h[Y_h^{m+1}])\bigr) \bigl((e_x^{m+1})_\rho\cdot(\vec n_h[Y_h^m]+\vec n_h[Y_h^{m+1}])\bigr) }{ |(Y_h^{m+1})_\rho| }\,\mathrm d\rho\\
 & \geq \frac{1}{2\tau}\Bigl(\left\|\partial_{s[Y_h^{m+1}]}e_x^{m+1}\cdot\vec n_h[Y_h^{m+1}]\right\|_{L^2(\Gamma[Y_h^{m+1}])}^2\\
 & \qquad\qquad - \left\|\partial_{s[Y_h^m]}e_x^m\cdot\vec n_h[Y_h^m]\right\|_{L^2(\Gamma[Y_h^m])}^2\Bigr) \\
 & \quad -\frac\tau2 \norm{\partial_{s[Y_h^m]}e_v^{m+1}}{L^2(\Gamma[Y_h^m])}^2 - C \tau \left\|V_Y^{m+1}\right\|_{W^{1, \infty}(I)}\left\|e_x^{m+1}\right\|_{1, h}^2 \\
 \end{aligned}
\end{equation}
Combine everything together, we obtain
\begin{equation}
 \begin{aligned}
 &\frac{1}{2\tau}\left(\left\|\partial_{s[Y_h^{m+1}]}e_x^{m+1}\cdot\vec n_h[Y_h^{m+1}]\right\|_{L^2(\Gamma[Y_h^{m+1}])}^2 - \left\|\partial_{s[Y_h^m]}e_x^m\cdot\vec n_h[Y_h^m]\right\|_{L^2(\Gamma[Y_h^m])}^2\right) \\
 & \qquad + c\tau |e_v^{m+1}|_{1, h}^2 + c\left\|e_v^{m+1} \cdot \widetilde{N}_h[X_h^m]\right\|_{0, h}^2 \\
 & \leq D_1^m \left\|e_x^m\right\|_{0, h} + C \left\|e_x^m\right\|_{1, h} \left\|e_v^{m+1}\right\|_{0, h}\\
 & \qquad + C \tau\left\|e_x^m\right\|_{1, h} \left\|e_v^{m+1}\right\|_{1, h} + C \tau \left\|e_x^{m+1}\right\|_{1, h}^2
 \end{aligned}
\end{equation}

\textit{Step 2: tangential error.} Take $\phi_h = I_h(\beta_h T_h[X_h^m])$ in the equation for $X_h^m$, $\phi_h = I_h(\beta_h T_h[Y_h^m])$ in the equation for $Y_h^m$, then subtracting gives the following tangential error equation
\begin{equation}\label{eq:error eq, tangential}
 \begin{aligned}
 &\tau\stiff{X_h^m}{e_v^{m+1}}{I_h(\beta_h T_h[X_h^m])}\\
 &=-d_h^{m+1}(I_h(\beta_h T_h[Y_h^m])) -\Bigl(\stiff{X_h^m}{X_h^m}{I_h(\beta_h T_h[X_h^m])}\\
 &\qquad\qquad -\stiff{Y_h^m}{Y_h^m}{I_h(\beta_h T_h[Y_h^m])}\Bigr)\\
 &\qquad-\tau\Bigl(\stiff{X_h^m}{V_Y^{m+1}}{I_h(\beta_h T_h[X_h^m])}\\
 &\qquad\qquad -\stiff{Y_h^m}{V_Y^{m+1}}{I_h(\beta_h T_h[Y_h^m])}\Bigr)\\
 &:= -d_h^{m+1}(I_h(\beta_h T_h[Y_h^m])) + J_{21}(I_h(\beta_h T_h[Y_h^m])) + J_{22}(I_h(\beta_h T_h[Y_h^m]))
 \end{aligned}
\end{equation}

Since $I_h(\beta_h T_h[Y_h^m]) = \pt{Y_h^m}I_h(\beta_h T_h[Y_h^m])$, for the RHS in \eqref{eq:error eq, tangential}, the defect estimate, Lemma \ref{lem:force}, and the standard perturbation estimate give
\begin{equation}
 \begin{aligned}
 &-d_h^{m+1}(I_h(\beta_h T_h[Y_h^m])) \leq D_2^m \left\|\pt{Y_h^m}I_h(\beta_h T_h[Y_h^m])\right\|_{1, h}\\
 &J_{21}(I_h(\beta_h T_h[Y_h^m])) \leq C h^2 \left\|e_x^m\right\|_{1, h} \left\|\beta_h\right\|_{1, h},\\
 & J_{22}(I_h(\beta_h T_h[Y_h^m])) \leq C \tau \left\|e_x^m\right\|_{1, h} \left\|\beta_h\right\|_{1, h}
 \end{aligned}
\end{equation}
Take $\beta_h = I_h(e_v^{m+1}\cdot T_h[X_h^m])$, $\stiff{X_h^m}{e_v^{m+1}}{I_h(\beta_h T_h[X_h^m])}$ in the LHS of \eqref{eq:error eq, tangential} becomes $\stiff{X_h^m}{e_v^{m+1}}{\pt{X_h^m}e_v^{m+1}}$. Now use Lemma \ref{lem:tancontrol}, we further obtain
\begin{equation}
 \begin{aligned}
 &\tau \stiff{X_h^m}{e_v^{m+1}}{\pt{X_h^m}e_v^{m+1}} \\
 & = \tau \stiff{X_h^m}{\pt{X_h^m}e_v^{m+1}}{\pt{X_h^m}e_v^{m+1}} \\
 & \qquad + \tau \stiff{X_h^m}{\pn{X_h^m}e_v^{m+1}}{\pt{X_h^m}e_v^{m+1}} \\
 & \geq c \tau\left\|\pt{X_h^m}e_v^{m+1}\right\|_{1, h}^2 - C \tau \left\|\pn{X_h^m}e_v^{m+1}\right\|_{0, h} \left\|\pt{X_h^m}e_v^{m+1}\right\|_{1, h}
 \end{aligned}
\end{equation}
Therefore, we derive the following control for the tangential part of the error
\begin{equation}
 \begin{aligned}
 &c \left\|\pt{X_h^m}e_v^{m+1}\right\|_{1, h}^2 \\
 & \leq \frac{D_2^m}{\tau} \left\|\pt{X_h^m}e_v^{m+1}\right\|_{1, h} + C \left(\frac{h^2}{\tau} +1\right) \left\|e_x^m\right\|_{1, h}\left\|\pt{X_h^m}e_v^{m+1}\right\|_{1, h} \\
 & \qquad  + C \left\|\pn{X_h^m}e_v^{m+1}\right\|_{0, h} \left\|\pt{X_h^m}e_v^{m+1}\right\|_{1, h}
 \end{aligned}
\end{equation}

\textit{Step 3: recover the full norm.} Since $e_x^{m+1} = e_x^m + \tau e_v^{m+1}$, we obtain
\begin{equation}
 \begin{aligned}
 \frac{1}{\tau} \left(\left\|e_x^{m+1}\right\|_{0, h}^2 - \left\|e_x^{m}\right\|_{0, h}^2\right) & = 2\int_{I} e_x^{m} e_v^{m+1} \, d\rho + \tau\int_{I} e_v^{m+1} e_v^{m+1}\, d\rho \\
 & \leq C \left\|e_x^m\right\|_{0, h} \left\|e_v^{m+1}\right\|_{0, h} + C\tau \left\|e_v^{m+1}\right\|_{0, h}^2
 \end{aligned}
\end{equation}

Using the equality $u\otimes u - w\otimes w = (u-w)\otimes w + u\otimes(u-w)$, we know that $\left\|\bigl(\pt{Y_h^{m+1}}-\pt{Y_h^m}\bigr)e_x^{m+1}\right\|_{1,h}\leq C\tau \norm{e_x^{m+1}}{1,h}$. Together with $\pt{Y_h^{m+1}}e_x^{m+1} = \pt{Y_h^{m}}e_x^{m} + \tau\pt{Y_h^{m}}e_v^{m+1} + (\pt{Y_h^{m+1}} - \pt{Y_h^{m}})e_x^{m+1}$, we obtain that

\begin{equation}
 \begin{aligned}
 &\frac{1}{\tau} \left(\left\|\pt{Y_h^{m+1}}e_x^{m+1}\right\|_{1, h}^2 - \left\|\pt{Y_h^m}e_x^{m}\right\|_{1, h}^2\right)\\
 & \leq C \left\|e_x^m\right\|_{1, h}\left\|\pt{Y_h^{m}}e_v^{m+1}\right\|_{1, h}\\
 & \qquad + C \left\|e_v^{m+1}\right\|_{1, h} \left\|(\pt{Y_h^{m+1}} - \pt{Y_h^{m}})e_x^{m+1}\right\|_{1, h} \\
 & \qquad + C\tau \left\|\pt{Y_h^m}e_v^{m+1}\right\|_{1, h}^2 + C \frac{1}{\tau} \left\|(\pt{Y_h^{m+1}} - \pt{Y_h^{m}})e_x^{m+1}\right\|_{1, h}^2 \\
 & \qquad + C \frac{1}{\tau}  \left\|e_x^m\right\|_{1, h} \left\|(\pt{Y_h^{m+1}} - \pt{Y_h^{m}})e_x^{m+1}\right\|_{1, h} \\
 & \leq C \left\|e_x^m\right\|_{1, h}\left\|\pt{Y_h^{m}}e_v^{m+1}\right\|_{1, h} + C\tau\left\|e_x^{m+1}\right\|_{1, h}\left\|e_v^{m+1}\right\|_{1, h}\\
 & \qquad + C\tau \left\|\pt{X_h^m}e_v^{m+1}\right\|_{1, h}^2 + C \tau \left\|e_x^{m+1}\right\|_{1, h}^2 + C \left\|e_x^m\right\|_{1, h} \left\|e_x^{m+1}\right\|_{1, h}\\
 & \qquad + C \tau \left\|e_x^m\right\|_{W^{1, \infty}(I)}^2 \left\|e_v^{m+1}\right\|_{1, h}^2.
 \end{aligned}
\end{equation}

Combine everything together, we deduce that
\begin{equation}
 \begin{aligned}
 &\frac{1}{2\tau}\left(\left\|\partial_{s[Y_h^{m+1}]}e_x^{m+1}\cdot\vec n_h[Y_h^{m+1}]\right\|_{L^2(\Gamma[Y_h^{m+1}])}^2 - \left\|\partial_{s[Y_h^m]}e_x^m\cdot\vec n_h[Y_h^m]\right\|_{L^2(\Gamma[Y_h^m])}^2\right) \\
 & \qquad +\frac{1}{2\tau} \left(\left\|\pt{Y_h^{m+1}}e_x^{m+1}\right\|_{1, h}^2 - \left\|\pt{Y_h^m}e_x^{m}\right\|_{1, h}^2\right)\\
 & \qquad + \frac{1}{2\tau} \left(\left\|e_x^{m+1}\right\|_{0, h}^2 - \left\|e_x^{m}\right\|_{0, h}^2\right)\\
 & \qquad + c\tau |e_v^{m+1}|_{1, h}^2 + c\left\|e_v^{m+1} \cdot \widetilde{N}_h[X_h^m]\right\|_{0, h}^2 + c \varrho\left\|\pt{X_h^m}e_v^{m+1}\right\|_{1, h}^2 \\
 & \leq D_1^m \left\|e_x^m\right\|_{0, h} + \varrho\frac{D_2^m}{\tau} \left\|\pt{X_h^m}e_v^{m+1}\right\|_{1, h}+ C \left\|e_x^m\right\|_{1, h} \left\|e_v^{m+1}\right\|_{0, h}\\
 & \qquad + C \tau\left\|e_x^m\right\|_{1, h} \left\|e_v^{m+1}\right\|_{1, h} + C \tau \left\|e_x^{m+1}\right\|_{1, h}^2\\
 & \qquad + C \varrho\left(\frac{h^2}{\tau} +1\right) \left\|e_x^m\right\|_{1, h}\left\|\pt{X_h^m}e_v^{m+1}\right\|_{1, h} \\
 & \qquad  + C \varrho\left\|\pn{X_h^m}e_v^{m+1}\right\|_{0, h} \left\|\pt{X_h^m}e_v^{m+1}\right\|_{1, h} \\
 & \qquad +  C \left\|e_x^m\right\|_{0, h} \left\|e_v^{m+1}\right\|_{0, h} + C\tau \left\|e_v^{m+1}\right\|_{0, h}^2 \\
 & \qquad + C \left\|e_x^m\right\|_{1, h}\left\|\pt{Y_h^{m}}e_v^{m+1}\right\|_{1, h} + C\tau\left\|e_x^{m+1}\right\|_{1, h}\left\|e_v^{m+1}\right\|_{1, h}\\
 & \qquad + C\tau \left\|\pt{X_h^m}e_v^{m+1}\right\|_{1, h}^2 + C \tau \left\|e_x^{m+1}\right\|_{1, h}^2 + C \left\|e_x^m\right\|_{1, h} \left\|e_x^{m+1}\right\|_{1, h}\\
 & \qquad + C \tau \left\|e_x^m\right\|_{W^{1, \infty}(I)}^2 \left\|e_v^{m+1}\right\|_{1, h}^2.
 \end{aligned}
\end{equation}
Use the fact $\tau = h^2$, and the induction assumption $\left\|e_x^k\right\|_{W^{1, \infty}(I)}\leq h^3$ for $k\leq m$ with choosing $\varrho>0$ small to absorb the RHS to the LHS, and use the norm equivalence, we derive that
\begin{equation}
 \begin{aligned}
 &\frac{1}{2\tau}\left(\left\|\partial_{s[Y_h^{m+1}]}e_x^{m+1}\cdot\vec n_h[Y_h^{m+1}]\right\|_{L^2(\Gamma[Y_h^{m+1}])}^2 - \left\|\partial_{s[Y_h^m]}e_x^m\cdot\vec n_h[Y_h^m]\right\|_{L^2(\Gamma[Y_h^m])}^2\right) \\
 & \qquad +\frac{1}{2\tau} \left(\left\|\pt{Y_h^{m+1}}e_x^{m+1}\right\|_{1, h}^2 - \left\|\pt{Y_h^m}e_x^{m}\right\|_{1, h}^2\right)\\
 & \qquad + \frac{1}{2\tau} \left(\left\|e_x^{m+1}\right\|_{0, h}^2 - \left\|e_x^{m}\right\|_{0, h}^2\right)\\
 & \qquad + c\tau |e_v^{m+1}|_{1, h}^2 + c\left\|e_v^{m+1} \cdot \widetilde{N}_h[X_h^m]\right\|_{0, h}^2 + c \left\|\pt{X_h^m}e_v^{m+1}\right\|_{1, h}^2 \\
 & \leq C(D_1^m)^2 + C \left(\frac{D_2^m}{\tau}\right)^2 + C \left\|e_x^m\right\|_{1, h}^2 + C \left\|e_x^{m+1}\right\|_{1, h}^2 \\
 & \leq C(\delta^m)^2 + C \Bigl(\left\|\partial_{s[Y_h^m]}e_x^{m}\cdot\vec n_h[Y_h^m]\right\|_{L^2(\Gamma[Y_h^m])}^2\\
 & \qquad\qquad + \left\|\pt{Y_h^{m}}e_x^{m}\right\|_{1, h}^2 + \left\|e_x^{m}\right\|_{0, h}^2\Bigr) \\
 & \qquad + C \Bigl(\left\|\partial_{s[Y_h^{m+1}]}e_x^{m+1}\cdot\vec n_h[Y_h^{m+1}]\right\|_{L^2(\Gamma[Y_h^{m+1}])}^2\\
 & \qquad\qquad + \left\|\pt{Y_h^{m+1}}e_x^{m+1}\right\|_{1, h}^2 + \left\|e_x^{m+1}\right\|_{0, h}^2\Bigr)
 \end{aligned}
\end{equation}

For sufficiently small $\tau$, discrete Gronwall and the norm equivalence give
\begin{equation}
 \begin{aligned}
 \max_{k\leq m+1}\left\|e_x^{k}\right\|_{H^1(\Gamma[Y_h^k])}^2 \leq C\left(\left\|e_x^{0}\right\|_{H^1(\Gamma[Y_h^0])}^2+\sum_{k\leq m}\tau (\delta^k)^2\right).
 \end{aligned}
\end{equation}
Since $e_x^0=0$ and $\delta^m\leq Ch^4$, the right-hand side is $Ch^8$, and we know $\left\|e_x^{m+1}\right\|_{H^1(\Gamma[Y_h^{m+1}])}\leq C h^4$.

Finally, the inverse inequality provides $\left\|e_x^{m+1}\right\|_{W^{1, \infty}} \leq C h^{7/2}$. Choosing $h_0$ sufficiently small, we know $\left\|e_x^{m+1}\right\|_{W^{1, \infty}} \leq h^3$, thus the induction holds, and we have \eqref{eq:error bound h4}.

Standard interpolation and Lemma \ref{lem:continuous-stability} give
\begin{equation}
 \begin{aligned}
 & \sup_m \left\|Y_h^m - Y^{[2]}(t_m)\right\|_{L^2(I)} + h \left\|Y_h^m - Y^{[2]}(t_m)\right\|_{H^1(I)} \leq C h^2 \\
 & \sup_{0\leq t\leq T}\left\|Y^{[0]}(t) - Y^{[2]}(t)\right\|_{H^1(I)} \leq C h^2.
 \end{aligned}
\end{equation}
\eqref{eq:error bound h2} is thus a direct result of \eqref{eq:error bound h4}.
\end{proof}

To obtain an optimal convergence rate, we need to control manifold distance with the $L^2$ error.
\begin{lemma}[Manifold distance]\label{lem:area}
For counterclockwise parametrization $X, Y:I\to\R^2$ of simple closed curves with $|X_\rho|, |Y_\rho| > c_0 > 0$,
\begin{equation}
 \Md(\Gamma[X],\Gamma[Y]) \leq\frac{1}{2} \left(\left\|X_\rho\right\|_{L^2(I)} + \left\|Y_\rho\right\|_{L^2(I)} \right)\left\|Y-X\right\|_{L^2(I)}.
\end{equation}
\end{lemma}
\begin{proof}
Let $\chi_{\Omega[X]}$ be the indicator function of $\Omega[X]$, and denote the winding number of $\Gamma[X]$ around $\vec{p}$ by $w_{\Gamma[X]}(\vec{p})$. Since $\Gamma[X]$ is a simple closed curve, and $X$ is a counterclockwise parametrization, we know that $\chi_{\Omega[X]} = w_{\Gamma[X]}$, a.e..

The manifold distance $\Md(\Gamma[X],\Gamma[Y])$ can be written as
\begin{equation}
 \Md(\Gamma[X],\Gamma[Y]) = |\Omega[X] \mathbin\triangle \Omega[Y]| = \int_{\mathbb{R}^2} |\chi_{\Omega[X]} - \chi_{\Omega[Y]}| = \int_{\mathbb{R}^2} |w_{\Gamma[X]} - w_{\Gamma[Y]}|.
\end{equation}

Let $c(\rho, \theta) = (1-\theta)Y(\rho) + \theta X(\rho)$. From \cite[(3)]{michor2006riemannian}, we know that
\begin{equation}
 \begin{aligned}
 \int_{\mathbb{R}^2} |w_{\Gamma[X]} - w_{\Gamma[Y]}| &\leq \int_0^1 \int_0^{1} |\partial_\rho c \times \partial_\theta c|\rm{d}\theta\rm{d}\rho \\
 & = \int_0^1 \int_0^{1} |((1-\theta)Y_\rho + \theta X_\rho) \times (Y - X)| \rm{d}\theta\rm{d}\rho  \\
 & \leq \int_0^1 \int_0^{1} (1-\theta)|Y_\rho \times (Y - X)| + \theta |X_\rho \times (Y - X) |\rm{d}\theta \rm{d}\rho \\
 & \leq \frac{1}{2} \left(\left\|Y_\rho\right\|_{L^2(I)} + \left\|X_\rho\right\|_{L^2(I)} \right)\left\|Y-X\right\|_{L^2(I)}
 \end{aligned}
\end{equation}
Combine the two equations, we get the desired result.
\end{proof}

\textit{Proof of the main theorem:} Apply the Lemma with $X_h^m$ and $Y^{[0]}(t_m)$. Since $Y^{[0]}$ differs from the solution $X$ of mean curvature flow only by a tangential velocity, which does not change the shape of the curve \cite{DeckelnickDziukElliott2005}, we have $\Gamma[Y^{[0]}(t_m)] = \Gamma[X(t_m)] = \Gamma(t_m)$. Moreover, Theorem~\ref{thm:main2} gives $\max_m\left(\norm{X_h^m-Y^{[0]}(t_m)}{L^2(I)}+h\norm{X_h^m-Y^{[0]}(t_m)}{H^1(I)}\right)\leq Ch^2$. Therefore, we have
\begin{equation}
 \begin{aligned}
 \Md(\Gamma[X_h^m],\Gamma[X(t_m)]) &\leq C\left(\left\|X_{h,\rho}^m\right\|_{L^2(I)} + \left\|Y^{[0]}_\rho(t_m)\right\|_{L^2(I)} \right)h^2 \\
 &\leq C \left(2\left\|Y^{[0]}_\rho(t_m)\right\|_{L^2(I)} + h \right)h^2 \leq C h^2.
 \end{aligned}
\end{equation}
The proof of Theorem~\ref{thm:main} is complete.

\section{Numerical results}\label{sec:numerics}
Extensive numerical experiments for the BGN method can be found in \cite{BGN2007b,BGN2008,BGN2020}. Here we only test the newly defined $Y^{[0]}$ and the defect.

We denote the BGN numerical solution at $t_m = m\tau$ by $X_h^m$, with $\tau=h^2$. We take the ellipse $X_0(\rho)=(2\cos2\pi\rho,\sin2\pi\rho)$ and $T=13/128\approx 0.1$. The comparison flow $Y^{[0]}$ in Section~\ref{sec:bea} is computed with BDF5 in time and $P^8$ elements in space.

Table~\ref{tab:ellipse} reports the defect and the tangential defect in \eqref{eq:zeroth-comparison-defects},
\begin{equation*}
 \sup_{\phi_h\in[S_h]^2}\frac{|d_h^{[0]}(\phi_h)|}{\norm{\phi_h}{0,h}}, \qquad \sup_{\beta_h\in S_h}\frac{\left|d_h^{[0]}\bigl(\Ih[\beta_hT_h[Y_h^{[0]}]]\bigr)\right|}{\norm{\beta_h}{0,h}},
\end{equation*}
and the errors $\norm{X_h^M-Y^{[0]}(T)}{L^2(I)}$ and $\norm{X_h^M-Y^{[0]}(T)}{H^1(I)}$.

\begin{table}[htbp]
\centering\small\setlength{\tabcolsep}{2pt}
\begin{tabular}{rcccccccc}
\toprule
$h$ & defect & Rate & tangential defect & Rate & $L^2$ error & Rate & $H^1$ error & Rate \\
\midrule
$2^{-4}$ & $1.5735\times10^{-1}$ & -- & $6.3868\times10^{-3}$ & -- & $9.0964\times10^{-2}$ & -- & $4.0227\times10^{-1}$ & -- \\
$2^{-5}$ & $4.7321\times10^{-2}$ & 1.73 & $8.0442\times10^{-4}$ & 2.99 & $2.2287\times10^{-2}$ & 2.03 & $2.0275\times10^{-1}$ & 0.99 \\
$2^{-6}$ & $1.2561\times10^{-2}$ & 1.91 & $5.7902\times10^{-5}$ & 3.80 & $5.6009\times10^{-3}$ & 1.99 & $1.0152\times10^{-1}$ & 1.00 \\
$2^{-7}$ & $3.1911\times10^{-3}$ & 1.98 & $3.7241\times10^{-6}$ & 3.96 & $1.4016\times10^{-3}$ & 2.00 & $5.0775\times10^{-2}$ & 1.00 \\
$2^{-8}$ & $8.0106\times10^{-4}$ & 1.99 & $2.3456\times10^{-7}$ & 3.99 & $3.5048\times10^{-4}$ & 2.00 & $2.5390\times10^{-2}$ & 1.00 \\
\bottomrule
\end{tabular}
\caption{Ellipse $X_0(\rho)=(2\cos2\pi\rho,\sin2\pi\rho)$ at $T=13/128$ with $\tau=h^2$.}\label{tab:ellipse}
\end{table}

We also take the nonconvex five-fold flower
\begin{equation*}
 X_0(\rho)=(1+0.1\cos10\pi\rho)(\cos2\pi\rho,\sin2\pi\rho)
\end{equation*}
with the same $T$. The results are reported in Table~\ref{tab:flower}.

\begin{table}[htbp]
\centering\small\setlength{\tabcolsep}{2pt}
\begin{tabular}{rcccccccc}
\toprule
$h$ & defect & Rate & tangential defect & Rate & $L^2$ error & Rate & $H^1$ error & Rate \\
\midrule
$2^{-4}$ & $1.4260\times10^{-1}$ & -- & $6.6757\times10^{-3}$ & -- & $6.4372\times10^{-2}$ & -- & $3.0404\times10^{-1}$ & -- \\
$2^{-5}$ & $4.0195\times10^{-2}$ & 1.83 & $6.8189\times10^{-4}$ & 3.29 & $1.4661\times10^{-2}$ & 2.13 & $1.4952\times10^{-1}$ & 1.02 \\
$2^{-6}$ & $1.0408\times10^{-2}$ & 1.95 & $5.0639\times10^{-5}$ & 3.75 & $3.7323\times10^{-3}$ & 1.97 & $7.4690\times10^{-2}$ & 1.00 \\
$2^{-7}$ & $2.6262\times10^{-3}$ & 1.99 & $3.3366\times10^{-6}$ & 3.92 & $9.3775\times10^{-4}$ & 1.99 & $3.7342\times10^{-2}$ & 1.00 \\
$2^{-8}$ & $6.5825\times10^{-4}$ & 2.00 & $2.1555\times10^{-7}$ & 3.95 & $2.3473\times10^{-4}$ & 2.00 & $1.8671\times10^{-2}$ & 1.00 \\
\bottomrule
\end{tabular}
\caption{Five-fold flower $X_0(\rho)=(1+0.1\cos10\pi\rho)(\cos2\pi\rho,\sin2\pi\rho)$ at $T=13/128$ with $\tau=h^2$.}\label{tab:flower}
\end{table}

In both examples, the defects are of orders $h^2$ and $h^4$, in agreement with \eqref{eq:zeroth-comparison-defects}. The errors are of order $h^2$ in $L^2$ and $h$ in $H^1$, in agreement with Theorem~\ref{thm:main2}.

\bibliographystyle{siamplain}
\bibliography{references}
\end{document}